\documentclass{scrartcl}

\KOMAoption{toc}{listof,bibliography,index}

		\usepackage{a4}   

		\usepackage{fancyhdr}
	
		\usepackage[ngerman,english]{babel} %american,italian,german
	
		\usepackage{eurosym}

		\usepackage{xspace}
	
		\usepackage{float} %Float-Handling mit Schalter H (gleiche Position wie im Skript)  %_timm_: braucht man nicht, oder?
		\usepackage{placeins} %Barriere für Float-Umgebungen erzeugen mit \FloatBarrier
	
		\usepackage{caption}
	
		\usepackage[fleqn]{amsmath}
		\usepackage{amsthm,amstext,amsfonts,amssymb}
		\usepackage{textcomp}
	
		\usepackage{graphicx}
	
		\usepackage{tikz}  
		\usetikzlibrary{arrows}
                \usetikzlibrary{shapes.geometric}
                \usetikzlibrary{shapes.symbols}
                \usetikzlibrary{positioning}
                \usetikzlibrary{calc}
                \usetikzlibrary{matrix} 
                \usetikzlibrary{fit}
                \usetikzlibrary{backgrounds}
	
		\usepackage{epsf}
	
		\usepackage{xcolor}
		\usepackage{makeidx}
		\makeindex %Muss vor begin{document}, sonst passiert nix
	
		\usepackage{cellspace}
	
		\usepackage[latin1]{inputenc}
		
		\usepackage[T1]{fontenc}

		\usepackage{lmodern}
	
		\usepackage{listings}
		\usepackage{booktabs,longtable,tabularx}
	
		\usepackage{subfigure}
		
		\usepackage{pdfpages}
		\usepackage{rotating}

		\usepackage{url}

		\usepackage{ifpdf} 

		\usepackage{fancyvrb}	

		\usepackage[intoc]{nomencl}

		  \makenomenclature

		\usepackage{enumerate}

		\usepackage{pgfplots}

		\newtheoremstyle{normalstyle}% ?name?
			{.5em}%	?Space above?
			{3pt}%	?Space below?
			{\normalfont\itshape}%	?Body font?
			{}%	?Indent amount?
			{\normalfont\bfseries}% ?Theorem head font?
			{.\newline}%	?Punctuation after theorem head?
			{.5em}%	?Space after theorem head?
			{}%	?Theorem head spec (can be left empty, meaning ?normal?)?
		\theoremstyle{normalstyle}
		\newtheorem{thm}{Theorem}[section]
		\newtheorem{lem}[thm]{Lemma}
                
		\newtheorem{cor}[thm]{Corollary}

                \theoremstyle{remark}
                \newtheorem{rem}[thm]{Remark}

                \theoremstyle{definition}

		\usepackage[bookmarks=true,
					bookmarksopen=true,
					bookmarksnumbered=true,
					]
				   {hyperref}

		\lstdefinelanguage{XMLSchema}
					{morekeywords={schema,element,annotation,appinfo,complexType,simpleType,choice,all,sequence},		
			sensitive=true,
			morestring=[b]",
		}
		
		\lstdefinelanguage{ASN1}
			{morekeywords={},		
			sensitive=true,
			morestring=[b]",
		}
		
		\renewcommand{\leq}		{\leqslant}
		\renewcommand{\geq}		{\geqslant}
		\renewcommand{\epsilon}	{\varepsilon}

	\ifpdf
	  \DeclareGraphicsExtensions{.jpg,.pdf,.png}   % for pdftex driver
	\else
	  \DeclareGraphicsExtensions{.eps}             % for dvips driver
	\fi
	
	\graphicspath{{images/}}

  \makeindex

\newcommand{\comment}[1]{} 

\usepackage[numbers]{natbib}

\usepackage{verbatim}

\newcommand{\ds}{\displaystyle}

\newcommand{\tr}{\operatorname{tr}}

\newcommand{\var}{\operatorname{var}}
\newcommand{\spann}{\operatorname{span}}

\newcommand{\Ord}{\mathcal{O}}

\newcommand{\Nn}{{\mathbb N}}

\newcommand{\Rr}{{\mathbb R}}
\newcommand{\Cc}{{\mathbb C}}

\newcommand{\vph}{\varphi}
\newcommand{\eps}{\varepsilon}

\newcommand{\Jj}{\mathbf{J}}
\newcommand{\ccc}{\mathbf{c}}
\newcommand{\ddd}{\mathbf{d}}

\newcommand{\Szwei}{\mathbb{S}_2}

\newcommand{\Pin}{\Pi_n}
\newcommand{\Pinm}{\Pi_{n}^m}

\newcommand{\Dx}[1]{\textrm{d}#1}

\newcommand{\MCos}[1]{\mathrm{M}_{\cos\theta} #1}
\newcommand{\Pnm}[1]{\mathrm{P}_n^m #1}
\newcommand{\Pnn}[1]{\mathrm{P}_n^0 #1}

\newcommand{\Nk}{N_k}

\newcommand{\Trm}{\mathrm{T}}

\newcommand{\expa}{^{(\alpha)}}

\def\be{\begin{equation}}
\def\ee{\end{equation}}
\def\ba{\begin{align*}}
\def\ea{\end{align*}}

\begin{document}

\title{An alternative to Slepian functions on the unit sphere - A space-frequency analysis based on localized spherical polynomials}
\author{Wolfgang Erb\thanks{Institute of Mathematics, University of Lübeck,
Ratzeburger Allee 160, 23562 Lübeck, Germany. erb@math.uni-luebeck.de}, \hspace{3cm} Sonja Mathias}

\date{15.05.2013}

\maketitle

\begin{abstract}
In this article, we present a space-frequency theory for spherical harmonics based on the spectral 
decomposition of a particular space-frequency operator. 
The presented theory is closely linked to the theory of ultraspherical polynomials on the one hand, 
and to the theory of Slepian functions on the $2$-sphere on the other. Results from both theories 
are used to prove localization and approximation properties of the new band-limited yet space-localized basis.
Moreover, particular weak limits related to the structure of the spherical harmonics provide information on the 
proportion of basis functions needed to approximate localized functions. Finally, a scheme for the fast
computation of the coefficients in the new localized basis is provided.

\end{abstract}
{\bf AMS Subject Classification}(2010): 42B05, 42C10, 45C05, 47B36 \\[0.5cm]
{\bf Keywords: Space-frequency analysis on the unit sphere, Slepian functions, spherical harmonics, ultraspherical polynomials, Jacobi matrices}

\section{Introduction}
\label{sec:introduction}

% Space-frequency analysis (or time-frequency analysis) denotes a mathematical subject area in which the behavior of functions is studied simultaneously
% in the space as well as in the frequency domain. In the core of every space-frequency theory
% one can find a mathematical formulation of the uncertainty principle stating 
% that a function cannot be arbitrarily well space-localized and band-limited at the same time (we refer to the survey article \citep{FollandSitaram1997} and the book \citep{HavinJoericke}
% for a general overview on uncertainty principles). Since signals obtained from physical measurement devices are usually bandlimited and one is often
% only interested in a specific region of the underlying space domain, techniques to analyze band-limited yet space-localized functions are necessary. 
% This task can be facilitated if the given functions can be expanded in a system of space- and frequency-localized basis functions. 
% The aim of this article is to provide and investigate such a space-localized and band-limited basis, if the underlying 
% domain is the unit sphere $\Szwei$. 

The roots of the space-frequency analysis studied in this article trace back to the works of Landau, Pollak and Slepian in the early 1960s 
(\citep{LandauPollak1961}, \citep{LandauPollak1962}, \citep{LandauWidom1980}, \citep{Slepian1964}, \citep{Slepian1978}, \citep{SlepianPollak1961}).
They considered band-limited functions on $\Rr$ that have maximal $L^2$-energy inside a given interval.
This optimization problem led them to the investigation of a particular integral equation with the
so-called prolate spheroidal wave functions as optimally space-concentrated eigenfunctions. These satisfy a series of remarkable analytic properties. Among others,
they form an orthogonal basis for the space of band-limited functions and emerge as solutions of the Helmholtz equation in prolate spheroidal coordinates.  

If the underlying domain is the unit sphere $\Szwei$, an analogous solution of the spatio-spectral concentration problem
was realized in \citep{GruenbaumLonghiPerlstadt1982}. The generalizations of the prolate spheroidal wave functions on $\Szwei$
were later on called Slepian functions and investigated thoroughly in several articles (see \citep{GruenbaumLonghiPerlstadt1982}, \citep{Simons2012},\citep{SimonsDahlenWieczorek2006}
and the references therein). The applications of the Slepian functions vary between spatio-spectral problems in geophysics (\citep{AlbertellaSansoSneeuw1999}, \citep{SimonsDahlenWieczorek2006}), 
planetary sciences (\citep{WieczorekSimons2005}) and medical imaging (\citep{Polyakov}). 

This article will consider an alternative approach to obtain a space-localized basis for spherical polynomials based on the idea of a 
preceding paper \citep{erb2013}. In \citep{erb2013}, a time-frequency analysis for orthogonal polynomials on the interval $[-1,1]$ based 
on the spectral decomposition of a particular time-frequency operator was studied. In the weighted Hilbert space $L^2([-1,1],w)$, this 
operator was defined as the composition $\Pnm{} \mathrm{M}_x\Pnm{}$ of a projection operator $\Pnm{}$ and a multiplication
operator $\mathrm{M}_x$. Using the orthogonal polynomials with respect to the weight function $w$, it was possible to show the unitary 
equivalence of $\Pnm{}\mathrm{M}_x\Pnm{}$ with the Jacobi matrix $\Jj_n^m$ of the orthonormal polynomial sequence. This connection made 
it possible to get explicit expressions for the eigenfunctions as well as to study their localization and approximation properties.

The idea of constructing a space-localized orthogonal basis as the eigenfunctions of a space-frequency operator linked to particular orthogonal polynomials
is now transferred to the setting of the unit sphere $\Szwei$. As a spherical analogue of the time-frequency operator 
in the one-dimensional polynomial setting we consider now the space-frequency operator $\Pnm{}\MCos{}\Pnm{}$, where
$\Pnm{}$ denotes the projection onto a space $\Pi_n^m$ of spherical harmonics and $\MCos{}$ the multiplication with $\cos \theta$. Here, $\Pnm{}$ plays
the role of a band-limiting operator, whereas the aim of $\MCos{}$ is to measure the space-localization of a function $f$ with respect to the geodetic distance $\theta$ from 
a particular point on the $2$-sphere (without loss of generality we will assume that this point is the north pole).
In Section \ref{sec:timefrequency}, we will further motivate the choice of the multiplication operator $\MCos{}$ and show that it is linked to a well-known 
uncertainty principle on the unit sphere (\citep{Erb2010, GohGoodman2004-2, NarcovichWard1996, RoeslerVoit1997}).

Using the multiplication operator $\MCos{}$ as a space-localization operator instead of a projection operator as in the Landau-Pollak-Slepian theory leads to some 
differences in the resulting space-frequency analysis. In order to compute the Slepian functions on the unit sphere efficiently, a second order differential operator
is needed that commutes with the respective space-frequency operator (see \citep{GruenbaumLonghiPerlstadt1982}, \citep{SimonsDahlenWieczorek2006}). 
In the space-frequency theory given in this paper such a differential operator is no longer needed. In Section \ref{sec:timefrequency}, 
it will turn out that the space-frequency operator $\Pnm{}\MCos{}\Pnm{}$ is unitarily equivalent to a tridiagonal block diagonal matrix consisting
of Jacobi matrices related to associated ultraspherical polynomials. This particular simple structure makes it possible to compute the eigenfunctions of the
operator $\Pnm{}\MCos{}\Pnm{}$ very efficiently. Moreover, due to the connection of the operator to the ultraspherical polynomials, it is possible to derive
a series of analytic properties for the spectrum as well as for the eigenfunctions of $\Pnm{}\MCos{}\Pnm{}$.  

In the Landau-Pollak-Slepian theory, the eigenvalues of the space-frequency operator indicate whether the corresponding eigenfunction is concentrated
in the examined sub-domain of $\Szwei$ (in this case the eigenvalue is close to one) or not (the eigenvalue is close to zero). The eigenvalues of
the space-frequency operator $\Pnm{}\MCos{}\Pnm{}$ examined in this article provide a different information on the space-localization of the eigenfunctions. They give a measure on
the mean geodetic distance from the north pole at which the corresponding eigenfunction is localized on $\Szwei$. In Section \ref{sec:eigenfunctions}, it will further
turn out that the eigenvalues are asymptotically uniformly distributed on $[-1,1]$. 

The article is structured as follows: In the next section, the necessary preliminaries concerning orthogonal expansions on the unit sphere and ultraspherical 
polynomials are given. In Section \ref{sec:timefrequency}, the space-frequency operator $\Pnm{}\MCos{}\Pnm{}$ is introduced. Its spectral 
decomposition forms the mathematical groundwork for the new space-frequency
analysis on $\Szwei$. The main result here is the spectral Theorem \ref{thm:spectraldecomposition} in which the eigenvalues and eigenfunctions of 
the space-frequency operator $\Pnm{}\MCos{}\Pnm{}$ are given explicitly. The localization and approximation properties of the eigenfunctions are studied
in Section \ref{sec:eigenfunctions}. Here, error bounds for the approximation of space-localized polynomials in $\Pi_n^m$ are given and the distribution 
of the eigenvalues is studied. In the last section, we will give some considerations regarding the computation of the coefficients in the new space-localized basis.
Due to the particular structure of the eigenfunctions and their relation to the ultraspherical polynomials, this can be done efficiently using algorithms based on the fast
Fourier transform.

\section{Preliminaries}
\label{sec:basics}
In this preliminary section, we summarize all necessary notation on spherical harmonics and orthogonal polynomials. 
A general overview on spherical harmonics and approximation theory on the unit sphere can be found in the monographs \citep{DaiXu,Freeden,Michel,Mueller} and in \citep[Section 2.1]{ConradPrestin2002}. 

On the unit sphere $\mathbb{S}_{2} := \lbrace \textbf{x} \in \mathbb{R}^3 \,: \, \Vert \textbf{x} \Vert_2 = 1 \rbrace$, every point $\textbf{x} \in \mathbb{S}_{2}$
can be written in spherical coordinates as
\begin{equation} \label{eq:conversion-cart-coord}
	\textbf{x} = (x_1, x_2, x_3) = (\sin\theta\cos\varphi,\sin\theta\sin\varphi,\cos\theta)
\end{equation}
where $\theta \in [0,\pi] $ denotes the polar angle and $\varphi \in [0,2\pi)$ the azimuth angle. 
The space of square-integrable functions on $\Szwei$ is defined as
\begin{equation}
\label{eq:def-L2}
	L^2(\mathbb{S}_2) := \left\{ f:\mathbb{S}_2 \rightarrow \mathbb{C} \textrm{ } \vert \left(\int_{\mathbb{S}_2} |f(\textbf{x})|^2\Dx{\omega(\textbf{x})}\right)^\frac{1}{2} < \infty \right\},
\end{equation}
where $\Dx{\omega(\textbf{x})}$ denotes the scalar surface element on $\Szwei$. In spherical coordinates, it can be written as $\Dx{\omega(\textbf{x})} = \sin\theta \,\Dx{\theta} \,\Dx{\varphi}$.
The inner product
\begin{equation}
	\label{eq:innerproduct}
	\langle f,g\rangle := \frac{1}{4\pi} \int_{\mathbb{S}_2} f(\textbf{x}) \overline{g(\textbf{x})} \, \Dx{\omega(\textbf{x})} = \frac{1}{4\pi} \int_0^{2\pi} \int_0^{\pi} f(\theta, \varphi) \overline{g(\theta, \varphi)} \sin\theta \, \Dx{\theta} \, \Dx{\varphi}
\end{equation}
turns $L^2(\mathbb{S}_2)$ into a Hilbert space. This space can be decomposed as $L^2(\Szwei) = \bigoplus_{l=0}^\infty \textrm{Harm}_l$, where $\textrm{Harm}_l$ denotes the $2l+1$ dimensional space 
spanned by the spherical harmonics $Y_l^k$, $-l \leq k \leq l$, of order $l \in \Nn_0$.  They can be written explicitly in spherical coordinates as
\begin{equation}
\label{eq:expr-sphericalharmonics}
	 Y_l^k(\theta, \varphi) := \sin^{|k|}\theta \; p_{l-|k|}^{(|k|)}(\cos\theta ) \; \textrm{e}^{ik\varphi}.
\end{equation}
Here, the polynomials $p_l^{(\alpha)}(x)$ denote the ultraspherical polynomials of degree $l$ with positive leading coefficient and orthonormal on $[-1,1]$ with respect to the inner product
	\begin{equation} 
	\label{eq:innerproduct-ultra}
	\langle f,g \rangle^{(\alpha)} := \frac{1}{2} \int_{-1}^1 f(x) \overline{g( x)} (1-x^2)^{\alpha} \Dx{x}.
	\end{equation}
For a detailed treatise on ultraspherical polynomials in the context of general orthogonal polynomials, we refer to \cite[Chapter 5]{Chihara}, \cite[Section 1.3.2]{Gautschi}, \cite[Chapter 4]{Ismail} and \cite[Chapter 4.7]{Szegoe}.

The orthonormal polynomials $p_l^{(\alpha)}$ satisfy the three-term recurrence relation
	\begin{align} 
	\label{eq:recursionorthonormal}
		b_{l+1}^{(\alpha)} p_{l+1}^{(\alpha)}(x) &= x p_l^{(\alpha)}(x) - 				b_l^{(\alpha)} p_{l-1}^{(\alpha)}(x), \quad l=0,1,2, \ldots \\
 		p_{-1}^{(\alpha)}(x) &= 0, \qquad p_0^{(\alpha)}(x) = \frac{1}{b_0^{(\alpha)}}, \notag
	\end{align}
	with the coefficients $b_l^{(\alpha)} = \left(\frac{l(l+2\alpha)}{(2l+2\alpha+1)(2l+2\alpha-1)}\right)^{\frac{1}{2}} $, $l>0$, and $b_0\expa = \left(\frac{\sqrt{\pi}}{2}\frac{\Gamma(\alpha+1)}{\Gamma(\alpha+\frac{3}{2})} \right)^{\frac{1}{2}} $.

Based on the three-term recurrence relation \eqref{eq:recursionorthonormal}, one observes that
	\begin{equation} 
	\label{eq:jacobian3t}
		\Jj(\alpha)_l \textbf{v}_{\alpha}(x) = x \textbf{v}_{\alpha}(x) - \begin{pmatrix}
			0 \\
		 	0 \\
		 	\vdots \\
		 	b_{l+1}p_{l+1}\expa(x) \\
		\end{pmatrix} 
	\end{equation}
	holds with $ \textbf{v}_{\alpha}(x) = (p_0\expa (x), p_1\expa (x),..., p_l\expa (x) )^T$ and the Jacobi matrix $\Jj(\alpha)_l$ defined as
	%\small
	\begin{equation}
	\label{eq:jacobi}
		\Jj(\alpha)_l := 	
		\begin{pmatrix}
			0 		& b_{1}^{(\alpha)} & 0 				  & 0 					& \cdots 			 	& 0 \\
			b_{1}^{(\alpha)} 	& 0 & b_{2}^{(\alpha)} & 0 					& \cdots 			 	& 0 \\
			0 					& b_{2}^{(\alpha)} & 0 & b_{3}^{(\alpha)} 	& \ddots 				& \vdots \\
			\vdots 				& \ddots 			 & \ddots			  & \ddots 				& \ddots  		 	& 0 \\
			0 					& \cdots 			 & 0				  &  b_{l-1}^{(\alpha)} & 0  & b_{l}^{(\alpha)} \\
			0 					& \cdots 			 &  \cdots       	  & 0 					& b_{l}^{(\alpha)}	& 0
		\end{pmatrix} .
	\end{equation}
	The associated ultraspherical polynomials $p_l^{(\alpha)}(x,m)$ are defined by the shifted recurrence relation 
(cf. \cite[Section 1.3.4]{Gautschi}, \cite[Section 5.7]{Ismail})
%, \cite[Section 2.10]{Ismail})
	\begin{align} 
	\label{eq:recursionassociatedsymmetric}
		b_{m+l+1}^{(\alpha)}\, p_{l+1}^{(\alpha)}(x,m) &= x \, p_l^{(\alpha)}(x,m) - b_{m+l}^{(\alpha)}\, p_{l-1}^{(\alpha)}(x,m),
		\quad l=0,1, \ldots , \\
 		p_{-1}^{(\alpha)}(x,m) &= 0, \qquad p_0^{(\alpha)}(x,m) = \frac{1}{b_m\expa}. \notag
	\end{align}
For $m =0$, the identity $p_l^{(\alpha)}(x,0) = \, p_l^{(\alpha)}(x)$ holds. In \citep[III, Section 4]{Chihara} and \cite[Section 1.3.4]{Gautschi}, the associated polynomials corresponding to an orthogonal polynomial sequence are called numerator polynomials.

For general $m \in \mathbb{N}_0$, equation \eqref{eq:jacobian3t} can be written as
\begin{equation} 
	\label{eq:generaljacobian3t}
		\Jj(\alpha)_{l+m}^m \textbf{v}_{\alpha}(x,m) = x \textbf{v}_{\alpha}(x,m) + \begin{pmatrix}
			0 \\
		 	0 \\
		 	\vdots \\
		 	b_{m+l+1}p_{l+1}\expa(x,m) \\
		\end{pmatrix} 
	\end{equation}
	with $ \textbf{v}_{\alpha}(x,m) = (p_0\expa (x,m), p_1\expa (x,m),..., p_l\expa (x,m) )^T$ and the truncated Jacobi matrix $ \Jj(\alpha)_l^m = \left(\Jj(\alpha)_l\right)_{i,j=m+1}^{l+1}$, $m \in \mathbb{N}_0$
\begin{rem} \label{rem:eigenvaluesJacobian}
	It follows from \eqref{eq:generaljacobian3t} that the eigenvalues of $ \Jj(\alpha)_{l+m}^m$ are exactly the $l+1$ roots $x_{\alpha,i}$, $1 \leq i \leq l+1$ of the associated ultraspherical polynomial $p_{l+1}\expa(x,m)$ with the eigenvectors	
\begin{equation*} \label{eq:eigenvectorsofJacobimatrix} \textbf{v}_{\alpha,i} := \textbf{v}_{\alpha}(x_{\alpha,i},m) = 
\left( p_0\expa (x_{\alpha,i},m), p_1\expa (x_{\alpha,i},m),..., p_l\expa (x_{\alpha,i},m) \right)^T .\end{equation*}
\end{rem}
Furthermore, the polynomials $p_l^{(\alpha)}(x,m)$ can be written as %(cf. \cite[Theorem 2.2.4]{Ismail})
\begin{align}
	p_l^{(\alpha)}(x,m) &= \frac{1}{b_{m}^{(\alpha)} b_{m+1}^{(\alpha)} \cdots b_{m+l}^{(\alpha)}}\det \left(x \mathbf{1}_{l} - \Jj(\alpha)_{m+l-1}^m \right), \label{eq:expressionbyjacobianassociated}
\end{align}
where $\mathbf{1}_{l}$ denotes the $l$-dimensional identity matrix. This is easily seen by verifying that the right hand side of equation \eqref{eq:expressionbyjacobianassociated} satisfies the recurrence relation \eqref{eq:recursionassociatedsymmetric}.

Finally, for $\alpha \in \Nn_0$, $\alpha \leq m $ and $x \neq y$ the following Christoffel-Darboux type identity holds (\cite[Lemma 3.1]{erb2012}):
%	\begin{align}
%	\label{eq:christ-darboux1}
%		\sum_{l=\alpha}^n & p_{l-\alpha}\expa (x) \ p_{l-\alpha}\expa(y)\\
%		& =  b_{n-\alpha +1}\expa \, \frac{p_{n-\alpha+1}\expa(x) \, p_{n-\alpha}\expa(y) -p_{n-\alpha+1}\expa(y) \, p_{n-\alpha}\expa(x)}{x-y} \notag 
%	\end{align}
	\begin{align} \label{eq:christ-darboux2}
		\sum_{l=m}^n & p_{l-\alpha}\expa (x) \ p_{l-m}\expa (y, m-\alpha) =  \frac{p_{m-\alpha-1}\expa(x)}{x-y} \\
		& + b_{n-\alpha +1}\expa \, \frac{p_{n-\alpha+1}\expa(x)\, p_{n-m}\expa(y,m-\alpha)-p_{n-\alpha}\expa(x)\, p_{n-m+1}\expa(y,m-\alpha)}{x-y}.\notag 
	\end{align}
For $\alpha = m$ the above formula reduces to the original Christoffel-Darboux formula, see \cite[Theorem 4.5]{Chihara}.

\section{Spectral analysis of the space-frequency operator} 
\label{sec:timefrequency}

The spherical harmonics $Y_l^k$, $m \leq l \leq n$, $-l \leq k \leq l$, form an orthonormal basis for the polynomial space
\[\Pi_n^m := \bigoplus_{l=m}^{n} \textrm{Harm}_l.\] 
Due to their harmonic nature, the $L^2$-mass of the spherical harmonics $Y_l^k$ is distributed over the whole $2$-sphere $\Szwei$. 
Spherical harmonics are therefore not well suited to decompose functions with mass concentrated in a specific sub-domain of $\Szwei$. 
The aim of this section is to obtain a set of space-localized basis functions in $\Pi_n^m$. 
To this end, we introduce and examine a particular space-frequency operator for $L^2$-functions on $\Szwei$ and derive its spectral decomposition. 
This spectral decomposition is the mathematical framework for a new space-localized basis in the space $\Pi_n^m$.

In comparison to other works (cf. \citep{GruenbaumLonghiPerlstadt1982, Miranian2004, SimonsDahlenWieczorek2006}) 
dealing with spatio-spectral concentration on the unit sphere, we use the multiplication operator $\MCos{}: L^2(\mathbb{S}_2) \to L^2(\mathbb{S}_2)$, defined by 
	\begin{equation*}
		(\MCos{f})(\theta, \varphi):=\cos\theta\, f(\theta, \varphi),
	\end{equation*}
to measure the space-localization of a function $f \in L^2(\Szwei)$. Introducing the mean value
	\begin{equation}
	\label{eq:generalizedmean}
		\epsilon(f) := \langle \MCos{f},f \rangle = \frac{1}{4\pi}\int_0^{2\pi} \! \int_0^{\pi} \cos\theta \vert f(\theta,\varphi)\vert^2 \sin\theta \, \Dx{\theta} \, \Dx{\varphi},
	\end{equation}
we can visualize the role of $\MCos{}$ as a descriptor of the localization of a function $f$ at the north pole of $\Szwei$. 
For a normalized function $f$ with $\Vert f \Vert = 1$, the mean value satisfies $-1 < \epsilon(f) <1$. 
If $\eps(f)$ is close to $1$, the mass of the $L^2(\mathbb{S}_2)$-density $f$ has to be situated in the region in which $\theta$ is close to zero, 
i.e., at the north pole of $\Szwei$, in such a way that the influence of $\cos \theta$ in the integral \eqref{eq:generalizedmean} is compensated. 
On the other hand, a mean value $\eps(f)$ close to $-1$ is an indicator for a mass concentration of $f$ at the south pole of $\Szwei$.
From now on, a normalized function $f$ is called \emph{space-localized} at the north or the south pole if its mean value $\eps(f)$ is close to $1$ or $-1$, respectively. 

\begin{rem}
The particular choice of $\cos \theta$ in the multiplication operator is motivated by the particular structure of the spherical harmonics on $\Szwei$. 
Since $\cos \theta = Y_{1,0}(\theta,\varphi)$, the value $\eps(f)$ can be considered as a first spherical moment of the $L^2$ density $f$. Further,
$\cos \theta$ is used in the Fisher-von Mises distribution on $\Szwei$ to measure the distance between a point on $\Szwei$ and the north pole (corresponding
to the mean point of the distribution), see \cite[Chapter 7]{GrafarendAwange}\footnote{\label{foot:1} We are very grateful to E. Grafarend for this hint.}. \\
The mean value $\eps(f)$ is also in further ways
connected to space localization on $\Szwei$. In \citep{Erb2010, NarcovichWard1996, RoeslerVoit1997}, the variance functional 
$\var_S(f) = \frac{1-\eps(f)^2}{\eps(f)^2}$ is used to measure the space localization of a function $f$ and is an essential part of an uncertainty principle on $\Szwei$. 
\end{rem}

To measure the frequency localization of a function $f \in L^2(\Szwei)$, we consider its projection onto the finite dimensional space $\Pi_n^m$. 
The corresponding projection operator is defined as 
	\begin{equation*}
		\Pnm{f}(\theta,\vph) := \sum_{l=m}^n \sum_{k=-l}^l \langle f, Y_l^k \rangle Y_l^k(\theta,\vph).
	\end{equation*}
The operator $\Pnm{}$ is bounded, self-adjoint and, since its range is a finite dimensional space, also compact. 
In the context of the space-frequency analysis discussed in this work, a polynomial $Q \in \Pi_n^m$ is called \emph{bandlimited}.

As the main mathematical object for a space-frequency analysis on $\Szwei$, we consider the composite operator
\begin{equation*}
	\Pnm{}\MCos{}\Pnm{}.
\end{equation*}
Due to the properties of $\MCos{}$ and $\Pnm{}$, also 
$\Pnm{}\MCos{}\Pnm{}$ is a compact and self-adjoint operator on $L^2(\Szwei)$.
By the Hilbert-Schmidt theorem for compact and self-adjoint operators \citep[Theorem VI.16]{ReedSimon1}, \citep[Theorem VI.3.2]{Werner}, we have the general spectral decomposition 
\begin{equation*}
	\Pnm{}\MCos{}\Pnm{f} = \sum_{j=1}^\infty \lambda_j \langle f, e_j \rangle e_j,  \qquad f \in  L^2(\mathbb{S}_2),
\end{equation*}
with $\lambda_j \in \Rr$, $j \in \Nn$, denoting the eigenvalues of $\Pnm{}\MCos{}\Pnm{}$ and $e_j$ the corresponding eigenfunctions. \\

The Hilbert-Schmidt theorem ensures that the eigenfunctions 
$\{ e_j \}_{j \in \Nn}$ form a complete orthonormal system of $L^2(\Szwei)$. Since every function 
in $L^2(\Szwei) \setminus \Pi_n^m$ is in the kernel of $\Pnm{}\MCos{}\Pnm{}$, it suffices to consider the operator $\Pnm{}\MCos{}\Pnm{}$ restricted 
to the polynomial space $\Pi_n^m$. The problem at hand now consists in calculating the eigenvalues and the corresponding eigenfunctions of 
$\Pnm{}\MCos{}\Pnm{}$ in the space $\Pi_n^m$. This is done by analyzing the behaviour of the operator $\Pnm{}\MCos{}\Pnm{}$ in the frequency domain, 
i.e., the space spanned by the expansion coefficients corresponding to the spherical harmonics. \\

To this end, we need some further notation. First of all, we consider the expansion of $Q \in \Pi_n^m$ in the basis of the spherical harmonics, i.e.,
\begin{equation} \label{equation-expansioninsphericalharmonics}
 Q(\theta, \varphi) = \sum_{l=m}^n \sum_{k=-l}^l c_{l,k} Y_l^k(\theta,\vph).
\end{equation} 
The dimension of $\Pi_n^m$ is given by
\[N_n^m = \dim \Pi_n^m = (n+1)^2-m^2 = (n+m+1)(n-m+1).\]
Now, we sort the $N_n^m$ coefficients $c_{l,k}$ of the expansion \eqref{equation-expansioninsphericalharmonics} according to the row-index $k$, as illustrated in Figure \ref{fig:arrangement-coefficients}.
Therefore, we introduce the coefficient vectors
\begin{align*}
\textbf{c}_{k} &:= (c_{m,k}, c_{m+1,k} \ldots, c_{n,k})^T,\quad -m \leq k \leq m,\\
\textbf{c}_{k} &:= (c_{|k|,k}, c_{|k|+1,k} \ldots, c_{n,k})^T,\quad m+1 \leq |k| \leq n,
\end{align*}
and the subspaces 
\[ \Pi_{n,k}^m = \spann \{Y_{l}^k:\; l = \max(|k|,m), \ldots, n\} \]
with the dimension \[ N_k = \dim \Pi_{n,k}^m = n - \max(|k|,m)+1. \] 
Further, we introduce the transition operators
	\begin{align*}
	\Trm_k & : \Cc^{N_k} \to \Pi_{n,k}^m: \quad \Trm_k \textbf{c}_{k} = \sum_{l=m}^n c_{l,k} Y_l^k,\quad -m \leq k \leq m,\\
	\Trm_k & : \Cc^{N_k} \to \Pi_{n,k}^m: \quad \Trm_k \textbf{c}_{k} = \sum_{l=|k|}^n c_{l,k} Y_l^k,\quad m+1 \leq |k| \leq n.
	\end{align*}
Finally, we denote the complete vector of coefficients by
\begin{align*}
\textbf{c} &:= (\ccc_n, \ccc_{n-1}, \cdots, \ccc_{-n})^T \in \Cc^{N_n^m}
\end{align*}
and introduce the overall transition operator
\[
\Trm : \Cc^{N_n^m} \longmapsto \Pi_n^m, \quad \Trm \ccc := \sum_{k=-n}^n \Trm_k \ccc_k = \sum_{l=m}^n \sum_{k=-l}^l c_{l,k} Y_l^k.
\]
Clearly, the linear operator $\Trm$ maps the canonical basis of $\Cc^{N_n^m}$ onto the orthonormal basis of $\Pi_n^m$ and is therefore an unitary operator.

\begin{figure}[h]
	\begin{center}
	\begin{tikzpicture}
		[scale=1,
		fi/.style={circle, thick, fill=black!20, minimum height = 1em, anchor = center},
		eigenf/.style ={circle, minimum height=1em, anchor = center}]
		%\tikzset{BarreStyle/.style = {opacity=.3,line width=8 mm,line cap=round,color=#1}};
		%\tikzset{highlight/.style = {opacity=.5,line width=2 mm,line cap = round, color = #1}};
		\tikzstyle{background}=[rectangle, fill=gray!20, inner sep=0cm, rounded corners=1mm]
		\tikzstyle{triangle}=[rectangle, fill=blue!30, inner sep=-0.1cm, rounded corners=1mm]

		\matrix [matrix of math nodes, 
					column sep={1cm,between origins}, 
					row sep={0.8cm,between origins},
					] (A)
		{
				& 		& 		& 		& 						& 						& |[eigenf]| c_{6,6}& \\ 
				& 		& 		& 		& 						& |[eigenf]| c_{5,5}& |[eigenf]| c_{6,5}& \\ 
				& 		& 		& 		& |[eigenf]| c_{4,4}& |[eigenf]| c_{5,4}& |[eigenf]| c_{5,6}& \\ 
				& 		& 		& |[fi]|& |[eigenf]| c_{4,3}& |[eigenf]| c_{5,3}& |[eigenf]| c_{6,3}& \\ 	
				& 		& |[fi]|& |[fi]|& |[eigenf]| c_{4,2}& |[eigenf]| c_{5,2}& |[eigenf]| c_{6,2}& \\ 
				& |[fi]|& |[fi]|& |[fi]|& |[eigenf]| c_{4,1}& |[eigenf]| c_{5,1}& |[eigenf]| c_{6,1}& \\ 		
		 |[fi]|	& |[fi]|& |[fi]|& |[fi]|& |[eigenf]| c_{4,0}& |[eigenf]| c_{5,0}& |[eigenf]| c_{6,0}& \\ 
		 		& |[fi]|& |[fi]|& |[fi]|& |[eigenf]| c_{4,\text{-}1}& |[eigenf]| c_{5,\text{-}1}& |[eigenf]| c_{6,\text{-}1}& \\ 
		 		& 		& |[fi]|& |[fi]|& |[eigenf]| c_{4,\text{-}2}& |[eigenf]| c_{5,\text{-}2}& |[eigenf]| c_{6,\text{-}2}& \\ 
		 		& 		& 		& |[fi]|& |[eigenf]| c_{4,\text{-}3}& |[eigenf]| c_{5,\text{-}3}& |[eigenf]| c_{6,\text{-}3}& \\ 
		 		& 		& 		& 		& |[eigenf]| c_{4,\text{-}4}& |[eigenf]| c_{5,\text{-}4}& |[eigenf]| c_{6,\text{-}4}& \\ 
		 		& 		& 		& 		&  						& |[eigenf]| c_{5,\text{-}5}& |[eigenf]| c_{6,\text{-}5}& \\ 
		 		& 		& 		& 		&  						& 						& |[eigenf]|c_{6,\text{-}6}& \\ 
		};
		% n- axis
                \node [left=60pt, label=left:{$\Pi_6^4$}] at (A-1-7) {};
		\node (x) [below=160pt, label=below:{l = 0}] at (A-7-1) {};
		\node (y) [below=160pt, label=below:{l = n = 6}] at (A-7-7) {};
		\node (z) [below=160pt, label=below:{l = m = 4}] at (A-7-5) {};
		\draw (x.west) -- (y.east);
		\draw (x.north) -- (x.south);
		\draw (y.north) -- (y.south);
		\draw (z.north) -- (z.south);
		
		% k- axis
		\node (a) [right=40pt] at (A-1-7) {};
		\node (b) [right=40pt] at (A-13-7) {};
		\node (c) [right=40pt, label=right:{k = 0}] at (A-7-7) {};
		\node (d) [right=40pt, label=right:{k = m = 4}] at (A-3-7) {};
		\node (e) [right=40pt, label=right:{k = -m = -4}] at (A-11-7) {};
		\node (f) [right=50pt, label=right:{$\ccc_2$}] at (A-5-7) {};
		\node (g) [right=50pt, label=right:{$\ccc_{\text{-}5}$}] at (A-12-7) {};
		\draw (a.north) -- (b.south);
		\draw (c.west) -- (c.east);
		\draw (d.west) -- (d.east);
		\draw (e.west) -- (e.east);

		\begin{pgfonlayer}{background}
        	\node [background, fit=(A-1-7) (A-13-7)] {};
		\node [background, fit=(A-2-6) (A-12-6)] {};
		\node [background, fit=(A-3-5) (A-11-5)] {};
                \node [triangle, fit=(A-5-5) (A-5-7)] {};
                \node [triangle, fit=(A-12-6) (A-12-7)] {};
                \node [background, left=78pt, fit = (A-1-7)]  {};
                \node [triangle, right=68pt, fit = (A-5-7)]  {};
                \node [triangle, right=68pt, fit = (A-12-7)]  {};
    	        \end{pgfonlayer}
							
	\end{tikzpicture}
	\caption[Visualization of the arrangement of the coefficients]{Visualization of the arrangement of the coefficients $c_{l,k}$ according to the structure of $\Pi_6^4$}
	\label{fig:arrangement-coefficients}
	\end{center}
\end{figure}

Now, we are able to represent the space-frequency operator $\Pnm{}\MCos{}\Pnm$ in the space of coefficient vectors.

\begin{lem} 
\label{lem:unitaryequivalence}
The operator $\Pnm{}\MCos{}\Pnm$ restricted to the polynomial space $\Pi_n^m$ is unitarily equivalent to the block diagonal matrix $\mathcal{J}_n^m \in \Cc^{N_n^m \times N_n^m}$ given by
\begin{equation*}
	\scalebox{0.45}{
	\begin{tikzpicture} [
		one/.style={regular polygon, regular polygon sides=4, fill=gray!20, inner sep = -20pt,  minimum height = 2.5cm,  rounded corners = 1mm, font = \huge},
		two/.style={regular polygon, regular polygon sides=4, fill=gray!20, inner sep = -20pt, minimum height = 3cm,  rounded corners = 1mm, font = \huge},
		m/.style={regular polygon, regular polygon sides=4, fill=gray!20, inner sep = -20pt, minimum height = 3.5cm, rounded corners = 1mm,  font = \huge}
		]
		%\tikzstyle{background}=[rectangle, fill=gray!20, inner sep=0.4cm, rounded corners=1mm]

		\matrix (A)[matrix of math nodes, %
					%execute at empty cell={\node[black!20]{0};},%
					column sep={0.1cm}, %
					row sep={-0.05cm},%
					left delimiter  = (,%
             		right delimiter = ),%
             		anchor=center,
             		ampersand replacement=\& %
             		] at (0,0)
		{ 
		% First Jacobian
			|[one]| \Jj(n)_0	\& \& \& \& \& \& \& \& \& \& \\
		% 2nd Jacobian
				\& 	|[two]|	\Jj(n\text{-}1)_1 \& \& \& \& \& \& \& \& \&\\	
				\& \& \ddots \& \& \& \& \& \& \& \& \& \&\\	
		% m-th Jacobian		
				\& \& \& |[m]| \Jj(m)_{n\text{-}m}^0 \& \& \& \& \& \& \&\\
				\& \& \& \& \ddots \& \& \& \& \& \& \\		
		% middle Jacobian
				\& \& \& \& \& |[m]| \Jj(0)_{n}^m \& \& \& \& \&\\
				\& \& \& \& \& \& \ddots \& \& \& \& \\		
		% k=-m Jacobian		
				\& \& \& \& \& \& \&|[m]| \Jj(m)_{n\text{-}m}^0\& \& \& \\
		% 2nd-to-last Jacobian
				\& \& \& \& \& \& \& \& \ddots \& \& \\	
				\& \& \& \& \& \& \& \& \& |[two]| \Jj(n\text{-}1)_1 \& \\					
		% Last Jacobian
				\& \& \& \& \& \& \& \& \& \& |[one]| \Jj(n)_0 \\					
		};
		\node (a) [left=50pt, font = \Huge] at (A.west) {$\mathcal{J}_n^m = $};				
	\end{tikzpicture}
	}
\end{equation*}
%\begin{equation}
%\mathcal{J}_n^m\!\! : = \!\diag\! \left( \! \Jj(n)_0,\! \Jj(n\text{-}1)_1,\! \cdots\!,\! \Jj(m)_{n\text{-}m}^0,\! \cdots\!,\! \Jj(0)_{n}^m\!,\! \cdots\!,\!\Jj(m)_{n\text{-}m}^0,\! \cdots\!,\! \Jj(n\text{-}1)_1,\! \Jj(n)_0\! \right),
%\end{equation}
where $\Jj( |k|)_n^m$ denote the Jacobi matrices of to the associated ultraspherical polynomials $p_l^{( |k|)}(x,m)$ defined in \eqref{eq:jacobi}.
More precisely, $\Trm^* (\Pnm{}\MCos{}\Pnm{}) \Trm = \mathcal{J}_n^m$ and 
\begin{align}
	 \Trm_k^{*} & ( \Pnm{} \MCos{} \Pnm )\Trm_k  = \Jj\big({ \textstyle |k|}\big)_{n-|k|}^{m-|k|}, \qquad -m \leq k \leq m, \label{eq:unitaryequivalencesmallk}\\
	 \Trm_k^{*} & ( \Pnm{} \MCos{} \Pnm )\Trm_k  = \Jj\big({ \textstyle |k|}\big)_{n-|k|}, \qquad m+1 \leq |k| \leq n. \label{eq:unitaryequivalencelargek}
\end{align}

Further, if $\ds Q = \sum_{l=m}^n \sum_{k=-l}^{l} c_{l,k} \, Y_l^k$, then
\begin{align} \label{eq:epsilonbyjacobian}
\epsilon(Q) &= \langle \Pnm \MCos \Pnm{Q}, Q\rangle \\ & = \sum_{k = -m}^{m}  \textbf{c}_{k}^H  \Jj\big({ \textstyle |k|}\big)_{n-|k|}^{m-|k|}\textbf{c}_{k} + \sum_{|k| = m+1}^{n}  \textbf{c}_{k}^H  \Jj \big({\textstyle |k|} \big)_{n-|k|} \textbf{c}_{k} = \ccc^H \mathcal{J}_n^m \ccc. \notag
\end{align}

\end{lem}

\begin{proof}
	We consider an arbitrary polynomial $Q \in \Pi_n^m$.
	Since the projection operator $\Pnm{}$ is self-adjoint, we get for the mean value
	\begin{equation}
	\label{eq:linkepsilpn}
		\epsilon(Q) = \langle \MCos{} Q, Q \rangle = \langle \MCos{} \Pnm{Q}, \Pnm{Q} \rangle = \langle \Pnm{} \MCos{} \Pnm{Q}, Q \rangle.
	\end{equation}
	Now using the expansion \eqref{equation-expansioninsphericalharmonics} of $Q$ in spherical harmonics and the transition operators $\Trm_k$, we obtain
	\begin{align*}
		\langle \Pnm{} \MCos{} \Pnm{Q}, Q \rangle = & \left\langle \Pnm{} \MCos{} \Pnm \left( \sum_{j=-n}^n \Trm_j \textbf{c}_j\right) , \! \sum_{k=-n}^n \Trm_k \textbf{c}_k \right\rangle \\
		= & \sum_{k=-n}^n  \sum_{j=-n}^n \langle \Pnm{} \MCos{} \Pnm{\Trm_j \textbf{c}_j}, \Trm_k \textbf{c}_k \rangle \\
		= & \sum_{k=-n}^n \sum_{j=-n}^n \underbrace{\langle \MCos{} \Trm_j \textbf{c}_j, \Trm_k \textbf{c}_k \rangle}_{\delta_{jk}\, \eps(\Trm_k \textbf{c}_k)} = \sum_{k=-n}^n \eps(\Trm_k \textbf{c}_k) \\
		= & \sum_{k=-n}^n \langle\Pnm{} \MCos{} \Pnm{\Trm_k \textbf{c}_k}, \Trm_k \textbf{c}_k \rangle,
	\end{align*}
	where the reduction in the third line is due to the orthogonality of the spherical harmonics $Y_n^j$ and $Y_{n'}^k$ for different indices $j \neq k$.
	Thus, the operator $\Pnm{} \MCos{} \Pnm{}$ on $\Pi_n^m$ has a reducible block structure and the length of the $2n+1$ blocks is given by the number
        $N_k$ of coefficients $c_{l,k}$ in the row $k$ (see also Figure \ref{fig:arrangement-coefficients}). To determine the behaviour of $\Pnm{} \MCos{} \Pnm{}$ on the 
subspaces $\Pi_{n,k}^m$, we consider the mean values $\eps(\Trm_k \textbf{c}_k)$ in more detail. Since the lengths $N_k$ of the subblocks differ, we have to consider two different cases and start with the case $-m \leq k \leq m$. By Definition \ref{eq:expr-sphericalharmonics} of the spherical harmonics $Y_l^k$, the mean value $\eps(\Trm_k \textbf{c}_k)$ can be expressed as
\begin{align*}
\epsilon(\Trm_k & \textbf{c}_k) = \frac{1}{4\pi} \int_0^{2\pi} \int_0^{\pi} \cos \theta \, |\Trm_k \textbf{c}_k (\theta,\varphi)|^2
\sin \theta\, \Dx{\theta} \, \Dx{\varphi} \\
&= \frac{1}{4\pi} \int_{0}^{2\pi} \int_{0}^{\pi} \left( \sum_{l=m}^n c_{l,k}
\cos \theta \, \sin^{|k|}\theta \,  {p}_{l-|k|}^{(|k|)}( \cos \theta ) \textrm{e}^{i k \varphi} \right) \\
& \hspace{2cm} \cdot \left(\sum_{l=m}^n \overline{c}_{l,k}
\sin^{|k|}\theta \, {p}_{l-|k|}^{(|k|)}
 ( \cos \theta )e^{-i k \varphi } \right) \sin \theta \,\Dx{\theta} \, \Dx{\varphi}. \\
&= \frac{1}{2} \!\int_{0}^{\pi} \!\! \left( \sum_{l=m}^n \!c_{l,k} 
\cos \theta {p}_{l-|k|}^{(|k|)}( \cos \theta  )\!\! \right) \!\!
\left( \sum_{l=m}^n \!  \overline{c}_{l,k} {p}_{l-|k|}^{(|k|)} ( \cos \theta )\!\! \right) \! \sin^{2|k|+1}\! \theta \Dx{\theta}. \\
\end{align*}
Now, using the three-term recurrence relation (\ref{eq:recursionorthonormal}) and the orthonormality of the polynomials ${p}_{l}^{(|k|)}$, we can conclude
\small
\begin{align*}
\epsilon(\Trm_k \textbf{c}_k)
&= \frac{1}{2} \int_{0}^{\pi} \Bigg( \sum_{l=m}^n  c_{l,k}
 \Big(b_{l-|k|}^{(|k|)} {p}_{l-|k|-1}^{(|k|)}( \cos \theta) +  b_{l-|k|+1}^{(|k|)} {p}_{l-|k|+1}^{(|k|)}(\cos \theta ) \Big) \Bigg)\\ 
 & \qquad \qquad \qquad \qquad \cdot \Big(\sum_{l=m}^n  \overline{c}_{l,k} \, {p}_{l-|k|}^{(|k|)} (\cos \theta)\Big) \sin^{2|k|+1}\theta \, \Dx{\theta}\\
 &=  \sum_{l=m+1}^{n} c_{l,k} \, \overline{c}_{l-1,k} \,  b_{l-|k|}^{|k|} + \sum_{l=m}^{n-1} c_{l,k} \, \overline{c}_{l+1,k}\,  b_{l-|k|+1}^{|k|} = 
 \textbf{c}_{k}^H \, \Jj \big({|k|}\big)_{n-|k|}^{m-|k|} \textbf{c}_{k} .
\end{align*}
In the case $m +1 \leq |k| \leq n$, we get by an analogous argumentation
\begin{align*}
\epsilon(\Trm_k \textbf{c}_k) 
&= \frac{1}{2} \!\int_{0}^{\pi} \!\! \left( \sum_{l=|k|}^n \!c_{l,k} 
\cos \theta {p}_{l-|k|}^{(|k|)}( \cos \theta  )\!\! \right) \!\!
\left( \sum_{l=|k|}^n \!  \overline{c}_{l,k} {p}_{l-|k|}^{(|k|)} ( \cos \theta )\!\! \right) \! \sin^{2|k|+1}\! \theta \Dx{\theta}. \\
&= \frac{1}{2} \int_{0}^{\pi} \Bigg( \sum_{l=|k|}^n  c_{l,k}
 \Big(b_{l-|k|}^{(|k|)} {p}_{l-|k|-1}^{(|k|)}( \cos \theta) +  b_{l-|k|+1}^{(|k|)} {p}_{l-|k|+1}^{(|k|)}(\cos \theta ) \Big) \Bigg)\\ 
 & \qquad \qquad \qquad \qquad \cdot \Big(\sum_{l=|k|}^n  \overline{c}_{l,k} \, {p}_{l-|k|}^{(|k|)} (\cos \theta)\Big) \sin^{2|k|+1}\theta \, \Dx{\theta}\\
 &=  \sum_{l=|k|+1}^{n} c_{l,k} \, \overline{c}_{l-1,k} \,  b_{l-|k|}^{|k|} + \sum_{l=|k|}^{n-1} c_{l,k} \, \overline{c}_{l+1,k}\,  b_{l-|k|+1}^{|k|} = 
 \textbf{c}_{k}^H \, \Jj \big({|k|}\big)_{n-|k|} \textbf{c}_{k} .
\end{align*}
In total, we can conclude
\[
\langle \Pnm \MCos \Pnm{Q}, Q\rangle = \sum_{k = -m}^{m}  \textbf{c}_{k}^H  \Jj\big({ \textstyle |k|}\big)_{n-|k|}^{m-|k|}\textbf{c}_{k} + \sum_{|k| = m+1}^{n}  \textbf{c}_{k}^H  \Jj \big({\textstyle |k|} \big)_{n-|k|} \textbf{c}_{k},
\]
and, thus, that equation \eqref{eq:epsilonbyjacobian} holds for any polynomial $Q \in \Pi_n^m$. Therefore, by the uniqueness theorem for self-adjoint operators (see \cite[Theorem 12.7]{RudinFunctionalAnalysis}),
the two operators $\Trm^* (\Pnm{}\MCos{}\Pnm{}) \Trm$ and $\mathcal{J}_n^m$ coincide and for the sub-operators $\Trm_k^* (\Pnm{}\MCos{}\Pnm{}) \Trm_k$ the relations \eqref{eq:unitaryequivalencesmallk} and \eqref{eq:unitaryequivalencelargek} hold.
\end{proof}

\begin{rem}
In Lemma \ref{lem:unitaryequivalence}, only the statement about the unitary equivalence of the two operators $\Pnm{}\MCos{}\Pnm$ and $\mathcal{J}_n^m$ is new. The characterization \eqref{eq:epsilonbyjacobian} of $\eps(f)$ with help of the operator $\mathcal{J}_n^m$ is not new and already stated and proven in a generalized form in \cite[Lemma 3.26]{ErbDiss} and \cite{Fernandez2007}. For the sake of completeness, we decided to formulate the proof here in a simplified form. 
\end{rem}

Now, we are able to state the spectral decomposition of the space-frequency operator $\Pnm{}\MCos{}\Pnm{}$ explicitly. 

% Spectral Theorem

\begin{thm}
\label{thm:spectraldecomposition}
	The space-frequency operator $\Pnm{}\MCos{}\Pnm{}$ on $L^2(\Szwei)$ has the spectral decomposition 
	\begin{align*}
		\Pnm{}\MCos{}\Pnm{f} &= \sum_{k = -n}^{n} \sum_{i=1}^{N_k} x_{|k|,i} \langle f, \psi_{k,i} \rangle \psi_{k,i}.
	\end{align*}
	For $-m \leq k \leq m$, the eigenvalues $x_{|k|,i}$, $1 \leq i \leq n-m+1$, denote the $n-m+1$ roots of the associated polynomial $p_{n-m+1}^{(|k|)}(x,m-|k|)$ and the eigenfunctions $\psi_{k,i}
\in \Pi_n^m$ have the explicit form
	\begin{align} 
	\label{eq:slepianfunctionsexplicit1}
	\psi_{k,i}(\theta,\varphi) = \kappa_{k,i} \, b_{n-|k|+1}^{(|k|)} & \, p_{n-m}^{(|k|)}(x_{|k|,i},m-|k|)  \frac{ \sin^{|k|}\theta \, p_{n-|k|+1}^{(|k|)}(\cos \theta) }{\cos \theta- x_{|k|,i}} \, e^{ik\varphi} \notag \\ 
	& + \kappa_{k,i} \, \frac{\sin^{|k|}\theta \, p_{m-|k|-1}^{(|k|)}(\cos \theta)}{\cos \theta - x_{|k|,i}} \, e^{ik\varphi},
	\end{align}
	with the normalizing constant
	\begin{equation} 
	\label{eq:normalizingconstant1}
		\kappa_{k,i} := \Big( \sum_{l=m}^{n} |p_{l-m}^{(|k|)}(x_{|k|,i},m-|k|)|^2 \Big)^{-\frac{1}{2}}.
	\end{equation}
	For $m+1 \leq |k| \leq n$, the eigenvalues $x_{|k|,i}$, $1 \leq i \leq n-|k|+1$, correspond to the $n-|k|+1$ roots of the polynomials $p_{n-|k|+1}^{(|k|)}(x)$ and the eigenfunctions 
        $\psi_{k,i}\in \Pi_n^m$ can be written as
	\begin{equation} 
	\label{eq:slepianfunctionsexplicit2}
	\psi_{k,i}(\theta,\varphi) = \kappa_{k,i} \, b_{n-|k|+1}^{(|k|)} \, p_{n-|k|}^{(|k|)}(x_{|k|,i}) \frac{  \sin^{|k|}\theta \, p_{n-|k|+1}^{(|k|)}(\cos \theta)}{\cos \theta - x_{|k|,i}} \, e^{ik\varphi},
	\end{equation}
	with the normalizing constant
	\begin{equation} 
	\label{eq:normalizingconstant2}
		\kappa_{k,i} := \Big( \sum_{l=|k|}^{n} |p_{l-|k|}^{(|k|)}(x_{|k|,i})|^2 \Big)^{-\frac{1}{2}}.
	\end{equation}
\end{thm}
\begin{proof}
        By Lemma \ref{lem:unitaryequivalence}, the operators $\Pnm{} \MCos{} \Pnm{}$ and $\mathcal{J}_n^m$ are unitarily equivalent and, thus, exhibit the same spectrum. In particular,
	the spectrum of $\Pnm{} \MCos{} \Pnm{}$ is composed of the eigenvalues of the Jacobi matrices $\Jj\big({ \textstyle |k|}\big)_{n-|k|}^{m-|k|}$, $0 \leq |k| \leq m$, and
        $\Jj\big({ \textstyle |k|}\big)_{n-|k|}$, $m+1 \leq |k| \leq n$. 

        To determine the single eigenfunctions, we consider first the case $-m\leq k \leq m$. If $\textbf{v}_{|k|,i}$ is an eigenvector of $\Jj\big({ \textstyle |k|}\big)_{n-|k|}^{m-|k|}$ corresponding to the eigenvalue $x_{|k|,i}$, then $\Trm_k \, \textbf{v}_{|k|,i}$ is an eigenfunction of $\Pnm{} \MCos{} \Pnm{}$, since
	\begin{align*}
		\Pnm{} \MCos{} \Pnm{} \,  \Trm_k \textbf{v}_{|k|,i}  & = \Trm_k \Trm_k^* \, \Pnm{} \MCos{} \Pnm{} \, \Trm_k \, \textbf{v}_{|k|,i}\\
		& = \Trm_k \, \Jj\big({ \textstyle |k|}\big)_{n-|k|}^{m-|k|}  \textbf{v}_{|k|,i} =x_{|k|,i} \, \Trm_k \, \textbf{v}_{|k|,i}.			
	\end{align*}
	By Remark \ref{rem:eigenvaluesJacobian}, the eigenvalues of $\Jj\big({ \textstyle |k|}\big)_{n-|k|}^{m-|k|}$ are exactly the $n-m+1$ roots $x_{|k|,i}$, $ i = 1, \cdots, n-m+1$, of the associated polynomial $p_{n-m+1}^{(|k|)}(x,m-|k|)$ with the corresponding eigenvectors 
	\[ \textbf{v}_{|k|,i}= (p_0^{(|k|)}(x_{|k|,i},m-|k|),p_1^{(|k|)}(x_{|k|,i},m-|k|),\dots ,p_{n-m}^{(|k|)}(x_{|k|,i},m-|k|))^T.\]
	Consequently, the normalized eigenfunctions of $\Pnm{} \MCos{} \Pnm{}$ can be written as
	\begin{align*}
		\psi_{k,i}(\theta, \varphi) = & \ \Trm_k \, \frac{\textbf{v}_{|k|,i}}{\Vert \textbf{v}_{|k|,i}\Vert_2} = \ \kappa_{k,i} \, \sum_{l=m}^n p_{l-m}^{(|k|)}(x_{|k|,i},m-|k|) \, Y_l^k(\theta, \varphi) \\
		 = & \ \kappa_{k,i} \, \sin^{|k|}\theta \, \textrm{e}^{ik\varphi} \sum_{l=m}^n p_{l-m}^{(|k|)}(x_{|k|,i},m-|k|) \, p_{l-|k|}^{(|k|)}(\cos\theta)\\
		 = & \ \kappa_{k,i} \, \sin^{|k|}\theta \,  \frac{b_{n-|k|+1}^{(|k|)} \, p_{n-m}^{(|k|)}(x_{|k|,i},m-|k|)\,  p_{n-|k|+1}^{(|k|)}(\cos\theta)}{\cos\theta - x_{|k|,i}} \,  \textrm{e}^{ik\varphi}  \\
		 & \quad + \kappa_{k,i} \, \sin^{|k|}\theta \, \frac{ p_{m-|k|-1}^{(|k|)}(\cos\theta)}{\cos\theta - x_{|k|,i}}\,  \textrm{e}^{ik\varphi} ,
	\end{align*}
	by using the Christoffel-Darboux type formula \eqref{eq:christ-darboux2} (bearing in mind that $p_{n-m+1}^{(|k|)}(x_{|k|,i},m-|k|)=0$) and defining the normalizing constant $\kappa_{k,i}$ as given in \eqref{eq:normalizingconstant1}.\\
	An analogous argumentation can be conducted for  $m+1 \leq |k| \leq n $. By Remark \ref{rem:eigenvaluesJacobian}, the eigenvalues of $\Jj\big({ \textstyle |k|}\big)_{n-|k|}$ are now the $n-|k|+1$ roots $x_{|k|,i}$, \mbox{$i = 1,..., n-|k|+1$}, of the ultraspherical polynomial $p_{n-|k|+1}^{(|k|)}(x)$. The corresponding eigenvectors are given by 
	\[ \textbf{v}_{|k|,i}= (p_0^{(|k|)}(x_{|k|,i}),p_1^{(|k|)}(x_{|k|,i}),\dots ,p_{n-|k|}^{(|k|)}(x_{|k|,i}))^T.\]
	Applying the transition operator $\Trm_k$ to this eigenvector and normalizing the result with $\kappa_{k,i}= \Vert \textbf{v}_{|k|,i} \Vert_2^{-1}$, we get the eigenfunctions
	\begin{align*}
		\psi_{k,i}(\theta, \varphi) 
		 = & \ \kappa_{k,i} \, \sin^{|k|}\theta \, \textrm{e}^{ik\varphi} \, \sum_{l=|k|}^n p_{l-|k|}^{(|k|)}(x_{|k|,i}) \, p_{l-|k|}^{(|k|)}(\cos\theta)\\
		 = & \ \kappa_{k,i} \, b_{n-|k|+1}^{(|k|)} \, p_{n-|k|}^{(|k|)}(x_{|k|,i}) \,\frac{\sin^{|k|}\theta p_{n-|k|+1}^{(|k|)}(\cos\theta)}{\cos\theta- x_{|k|,i}} \, \textrm{e}^{ik\varphi}
	\end{align*}
	again by virtue of equation \eqref{eq:christ-darboux2}. In total, we have found $\sum_{k=-n}^n N_k = N_n^m$ different, and therefore all, eigenfunctions of $\Pnm{} \MCos{} \Pnm{}$ in $\Pi_m^n$. 
\end{proof}

\begin{rem}
The spectral Theorem \ref{thm:spectraldecomposition} for the operator $\Pnm{}\MCos{}\Pnm{}$ on the unit sphere is completely novel in this work. However, there exist related versions of 
$\Pnm{}\MCos{}\Pnm{}$ and Theorem \ref{thm:spectraldecomposition} in simpler settings. For orthogonal polynomials on the real line or on the unit circle the spectral analysis of the respective operator leads to deep results for the theory of orthogonal polynomials itself (see \citep{Simon2011}). The space-frequency analysis of such an operator for orthogonal polynomials on $[-1,1]$ is studied in \citep{erb2013}.
\end{rem}

\begin{figure}[h]
	\begin{center}
	\begin{tikzpicture}
		[scale=1,
		fi/.style={circle, thick, fill=black!20, minimum height = 1em, anchor = center},
		eigenf/.style ={circle, minimum height=1em, anchor = center}]
		%\tikzset{BarreStyle/.style = {opacity=.3,line width=8 mm,line cap=round,color=#1}};
		%\tikzset{highlight/.style = {opacity=.5,line width=2 mm,line cap = round, color = #1}};
		\tikzstyle{background}=[rectangle, fill=gray!20, inner sep=0cm, rounded corners=1mm]
		\tikzstyle{triangle}=[rectangle, fill=blue!30, inner sep=-0.1cm, rounded corners=1mm]

		\matrix [matrix of math nodes, 
					column sep={1cm,between origins}, 
					row sep={0.8cm,between origins},
					] (A)
		{
					& 		& 		& 		& 						& 						& |[eigenf]| \psi_{6,1}& \\ 
					& 		& 		& 		& 						& |[eigenf]| \psi_{5,2}& |[eigenf]| \psi_{5,1}& \\ 
					& 		& 		& 		& |[eigenf]| \psi_{4,3}& |[eigenf]| \psi_{4,2}& |[eigenf]| \psi_{4,1}& \\ 
					& 		& 		& |[fi]|& |[eigenf]| \psi_{3,3}& |[eigenf]| \psi_{3,2}& |[eigenf]| \psi_{3,1}& \\ 	
					& 		& |[fi]|& |[fi]|& |[eigenf]| \psi_{2,3}& |[eigenf]| \psi_{2,2}& |[eigenf]| \psi_{2,1}& \\ 
					& |[fi]|& |[fi]|& |[fi]|& |[eigenf]| \psi_{1,3}& |[eigenf]| \psi_{1,2}& |[eigenf]| \psi_{1,1}& \\ 		
			 |[fi]|	& |[fi]|& |[fi]|& |[fi]|& |[eigenf]| \psi_{0,3}& |[eigenf]| \psi_{0,2}& |[eigenf]| \psi_{0,1}& \\ 
			 		& |[fi]|& |[fi]|& |[fi]|& |[eigenf]| \psi_{-1,3}& |[eigenf]| \psi_{-1,2}& |[eigenf]| \psi_{-1,1}& \\ 
			 		& 		& |[fi]|& |[fi]|& |[eigenf]| \psi_{-2,3}& |[eigenf]| \psi_{-2,2}& |[eigenf]| \psi_{-2,1}& \\ 
			 		& 		& 		& |[fi]|& |[eigenf]| \psi_{-3,3}& |[eigenf]| \psi_{-3,2}& |[eigenf]| \psi_{-3,1}& \\ 
			 		& 		& 		& 		& |[eigenf]| \psi_{-4,3}& |[eigenf]| \psi_{-4,2}& |[eigenf]| \psi_{-4,1}& \\ 
			 		& 		& 		& 		&  						& |[eigenf]| \psi_{-5,2}& |[eigenf]| \psi_{-5,1}& \\ 
			 		& 		& 		& 		&  						& 						& |[eigenf]| \psi_{-6,1}& \\ 
		};
		% n- axis
                \node [left=60pt, label=left:{$\Pi_6^4$}] at (A-1-7) {};
		\node (x) [below=160pt] at (A-7-1) {};
		\node (y) [below=160pt, label=below:{i = 1}] at (A-7-7) {};
		\node (z) [below=160pt, label=below:{i = 3}] at (A-7-5) {};
		\draw (x.west) -- (y.east);
		\draw (y.north) -- (y.south);
		\draw (z.north) -- (z.south);
		
		% k- axis
		\node (a) [right=40pt] at (A-1-7) {};
		\node (b) [right=40pt] at (A-13-7) {};
		\node (c) [right=40pt, label=right:{k = 0}] at (A-7-7) {};
		\node (d) [right=40pt, label=right:{k = m = 4}] at (A-3-7) {};
		\node (e) [right=40pt, label=right:{k = -m = -4}] at (A-11-7) {};
		\node (f) [right=50pt, label=right:{$(\Trm_2 \textbf{v}_{2,3}, \Trm_2 \textbf{v}_{2,2}, \Trm_2 \textbf{v}_{2,1}) $}] at (A-5-7) {};
		\node (g) [right=50pt, label=right:{$(\Trm_{\text{-}5} \textbf{v}_{5,2}, \Trm_{\text{-}5} \textbf{v}_{5,1})$}] at (A-12-7) {};
		\draw (a.north) -- (b.south);
		\draw (c.west) -- (c.east);
		\draw (d.west) -- (d.east);
		\draw (e.west) -- (e.east);

		\begin{pgfonlayer}{background}
        	\node [background, fit=(A-1-7) (A-13-7)] {};
		\node [background, fit=(A-2-6) (A-12-6)] {};
		\node [background, fit=(A-3-5) (A-11-5)] {};
                \node [triangle, fit=(A-5-5) (A-5-7)] {};
                \node [triangle, fit=(A-12-6) (A-12-7)] {};
                \node [background, left=78pt, fit = (A-1-7)]  {};
                \node [triangle, right=170pt, fit = (A-5-3) (A-5-7)]  {};
                \node [triangle, right=98pt, fit = (A-12-6) (A-12-7)]  {};
                \node [triangle, right=138pt, fit = (A-12-6) (A-12-7)]  {};
    	        \end{pgfonlayer}
							
	\end{tikzpicture}
	\caption[Visualization of the arrangement of the eigenfunctions]{Visualization of the arrangement of the eigenfunctions $\psi_{k,i}$ according to the structure of $\Pi_6^4$}
	\label{fig:arrangement-eigenfunctions}
	\end{center}
\end{figure}

\section{Localization and approximation properties of the eigenfunctions}
\label{sec:eigenfunctions}

The Hilbert-Schmidt theorem ensures that the eigenfunctions $\psi_{k,i}$ derived in Theorem \ref{thm:spectraldecomposition} form an orthonormal basis in the space $\Pi_n^m$. 
In this section, we will investigate some more properties of the eigenfunctions $\psi_{k,i}$ related to their space localization on the unit sphere $\Szwei$.

First of all, it follows from the definition of the $\psi_{k,i}$ as normalized eigenfunctions of the operator $\Pnm{} \MCos{} \Pnm{} $ that the mean value $\epsilon(\psi_{k,i})$
coincides with the eigenvalue $x_{|k|,i}$, i.e., 
\[\epsilon(\psi_{k,i}) = \langle \Pnm{} \MCos{} \Pnm{\psi_{k,i}},\psi_{k,i} \rangle = x_{|k|,i} \langle \psi_{k,i},\psi_{k,i} \rangle = x_{|k|,i}.\]
Using the correspondence of the eigenvalues $x_{|k|,i}$ with the zeros of the associated ultraspherical polynomial $p_{n-m+1}^{(|k|)}(x,m-|k|)$ or $p_{n-|k|+1}^{(|k|)}(x)$, 
respectively, we can describe the localization regions of the single eigenfunctions $\psi_{k,i}$. 

Sorting all eigenvalues $x_{|k|,i}$, $0 \leq |k| \leq n$, $1 \leq i \leq N_k$, from a maximal eigenvalue $x_{\textrm{max}}$ to a minimal eigenvalue $x_{\textrm{min}}$ 
results in a hierarchy in the sequence of the eigenfunctions concerning the space localization measured by $\epsilon(f)$. In this sense, the eigenfunction 
corresponding to $x_{\textrm{max}}$ is optimally localized at the north pole among all eigenfunctions. Then, examining the orthogonal 
complement $\Pi_n^m \ominus \text{span}\{\psi_{\text{max}} \}$, the optimally localized eigenfunction is $\psi_{\text{max}-1} $. Accordingly,
one obtains a chain of eigenfunctions in which the $j$-th. element is better localized with respect to the the north pole than the subsequent $(j-1)$-th. element. 

There exists a series of relations between the different eigenvalues. In the following we list a few important ones.
\begin{itemize}
\item[a)] For a fixed row-number $k$, the eigenvalues $x_{|k|,i}$, $1 \leq i \leq N_k$, are all in the interior of $[-1,1]$ and pairwise distinct.
This is a standard result from the theory of orthogonal polynomials (see \citep[I. Theorem 5.2]{Chihara}). In the following, we order the $N_k$ zeros 
in decreasing size such that 
\begin{align*}
	x_{k,1} &> x_{k,2} >\ldots > x_{k,N_k-1} > x_{k,N_k}.
\end{align*}
\item[b)] For $m \leq |k| < n$, the zeros $x_{|k|,i}$ and $x_{|k|+1,i}$ are interlacing (cf. \citep[I. Theorem 5.3]{Chihara}), i.e. 
\begin{equation*}
	x_{k,i} > x_{k+1,i} > x_{k,i+1}, \quad 1 \leq i \leq N_k-1.
\end{equation*}
\item[c)] If $|k| < m$, then $x_{k,1} > x_{k+1,1}$ and $x_{k,n-m+1} < x_{k+1,n-m+1}$. In particular, 
this statement implies that $x_{\max} = x_{0,1}$ and $x_{\min} = x_{0,n-m+1}$. The polynomial $Q \in \Pi_n^m$, $\|Q\| = 1$, 
that maximizes the functional $\eps(Q)$ and that is optimally localized at the north pole is therefore given by the eigenfunction 
$\psi_{0,1}$. By the same argumentation, the eigenfunction that is best localized at the south pole is given by $\psi_{0,n-m+1}$. 
These results were deduced in \citep{ErbDiss, ErbTookos2011, Fernandez2007}. Polynomials that are optimally localized with respect to
more general localization functionals, were studied in \cite{Michel2011}.
\end{itemize}

The ordering of the different eigenfunctions for the polynomial space $\Pi_6^4$ is illustrated in Figure \ref{fig:arrangement-eigenfunctions}. Figure \ref{fig:plot1} shows some examples of the real part of the eigenfunctions. 

\begin{figure}[H]
	\centering

	\subfigure[$\psi_{0,1}$, $x_{0,1} = 0.9974$]{\includegraphics[width=0.38\textwidth]{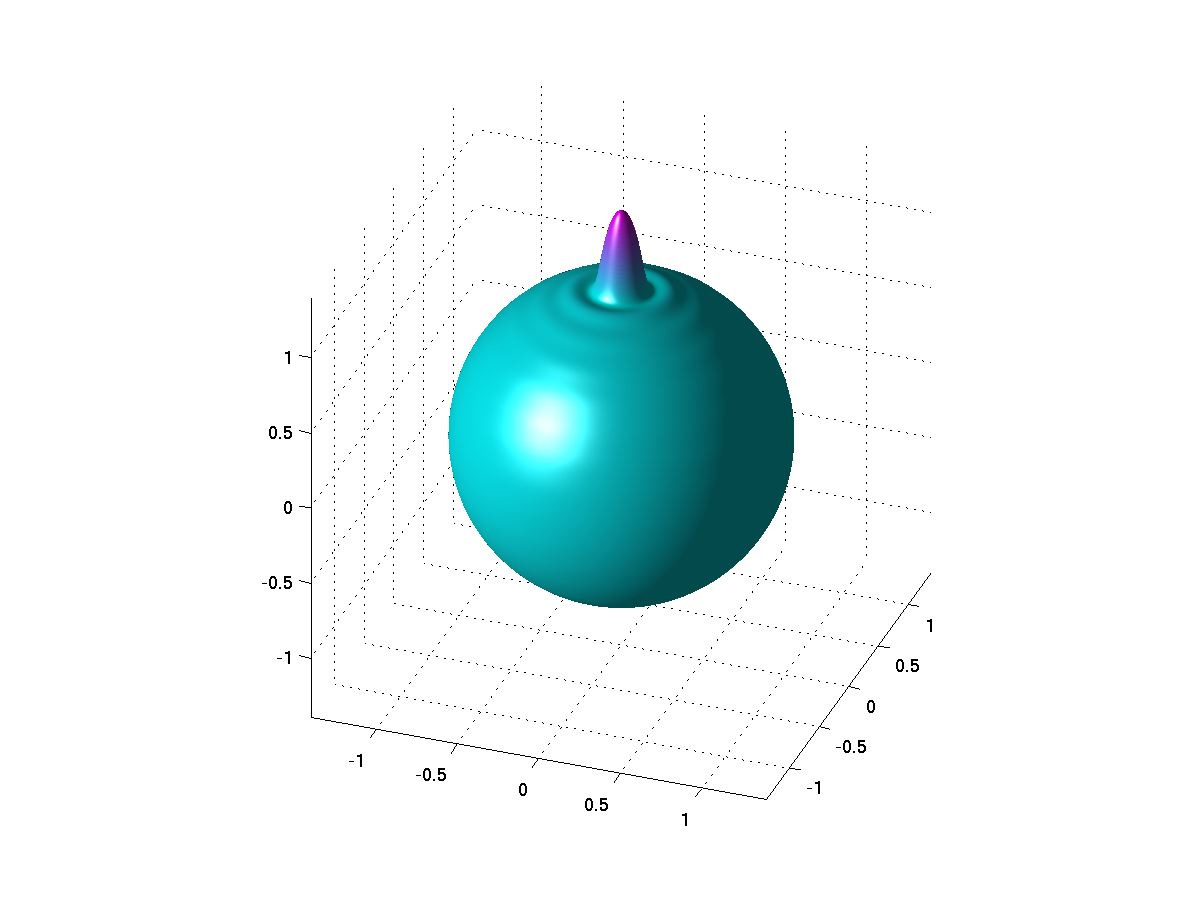}} \hspace{-12mm}
	\subfigure[$\psi_{0,8}$, $x_{0,8} = 0.7472$]{\includegraphics[width=0.38\textwidth]{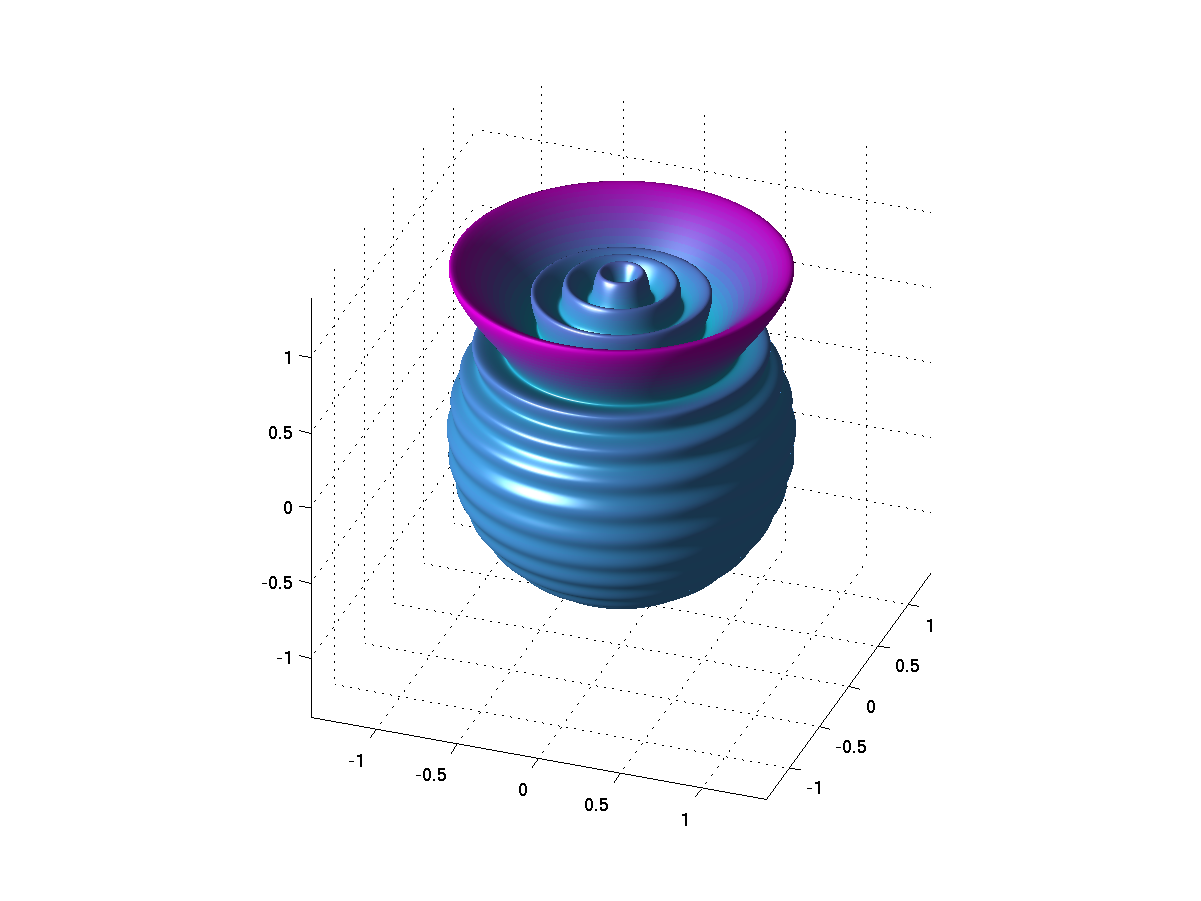}} \hspace{-12mm}
	\subfigure[$\psi_{0,16}$, $x_{0,16} = 0.0936$]{\includegraphics[width=0.38\textwidth]{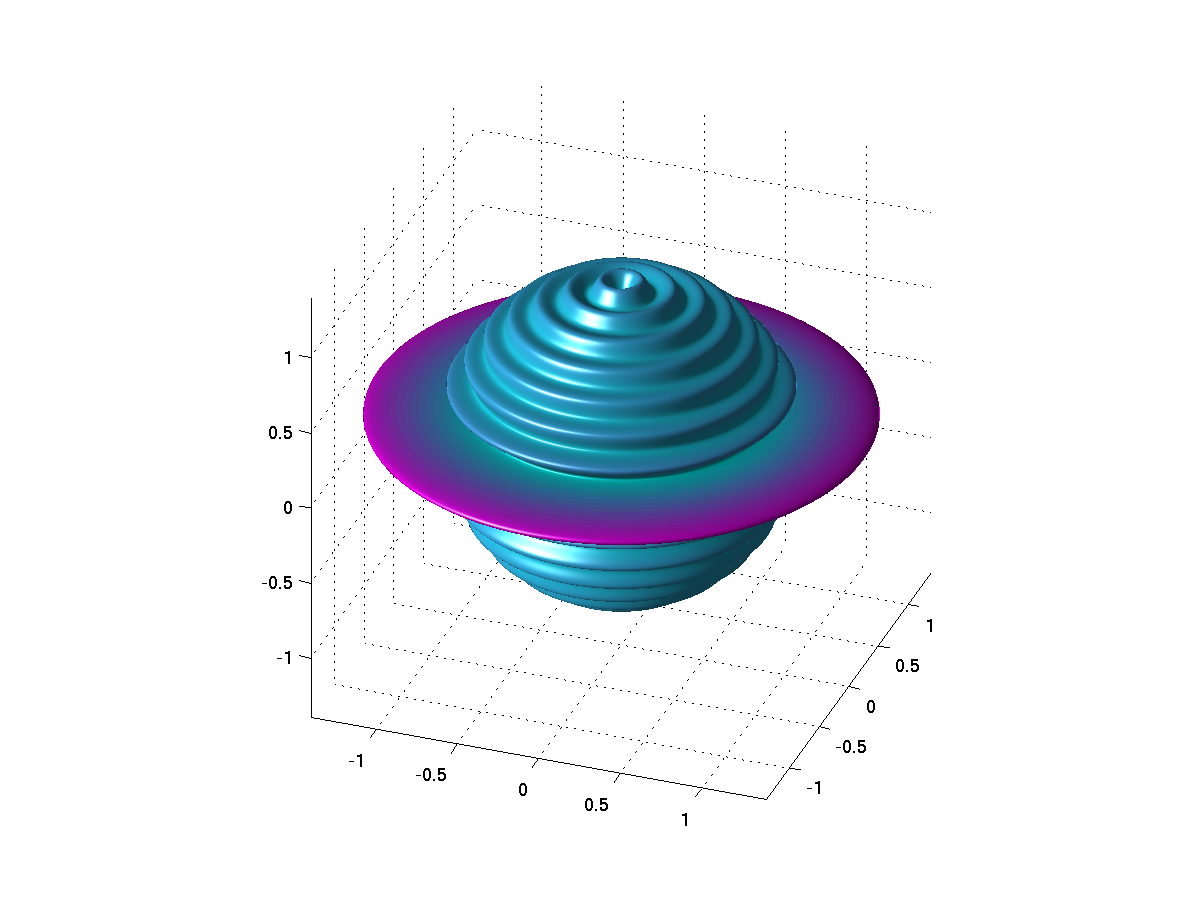}} \\

	\subfigure[$\psi_{16,1}$, $x_{16,1} = 0.7921$]{\includegraphics[width=0.38\textwidth]{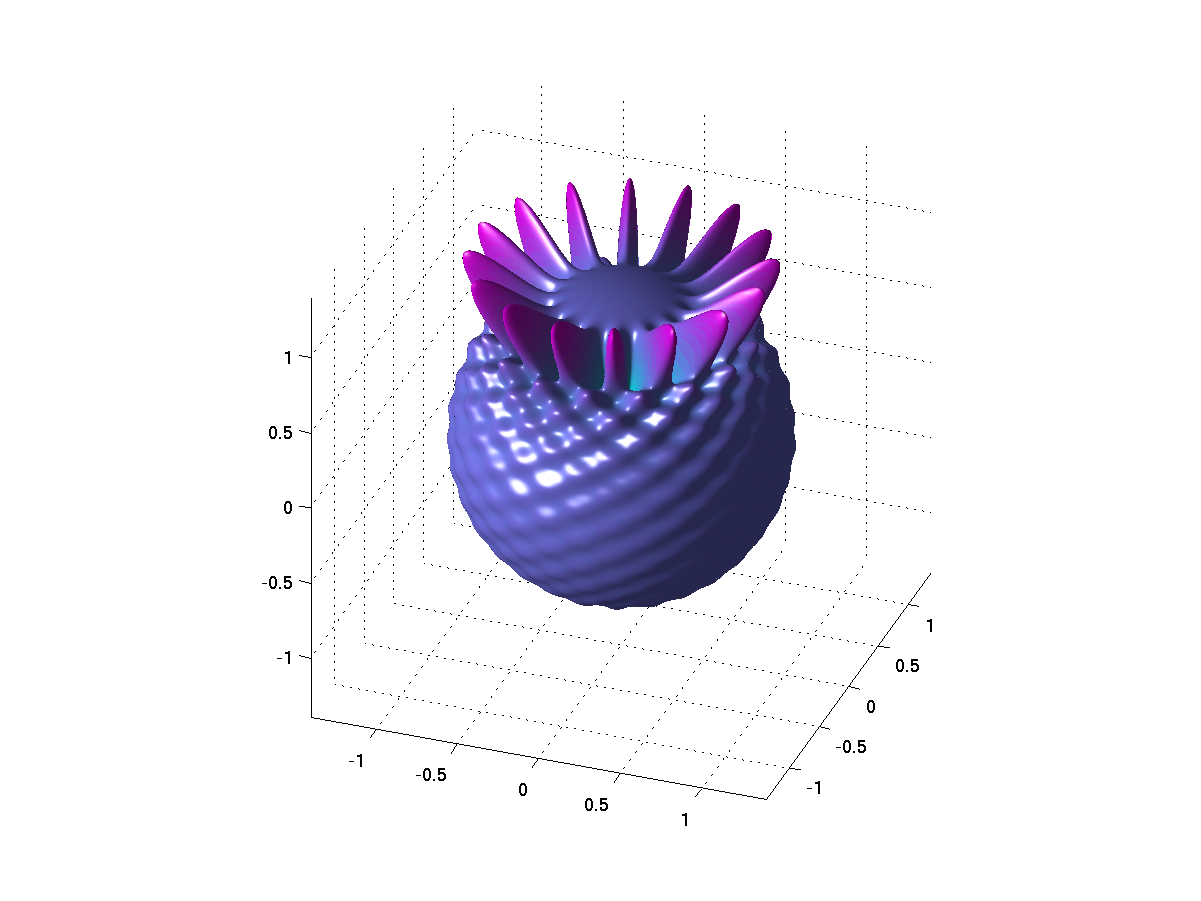}} \hspace{-12mm}
	\subfigure[$\psi_{16,8}$, $x_{16,8} = 0.1066$]{\includegraphics[width=0.38\textwidth]{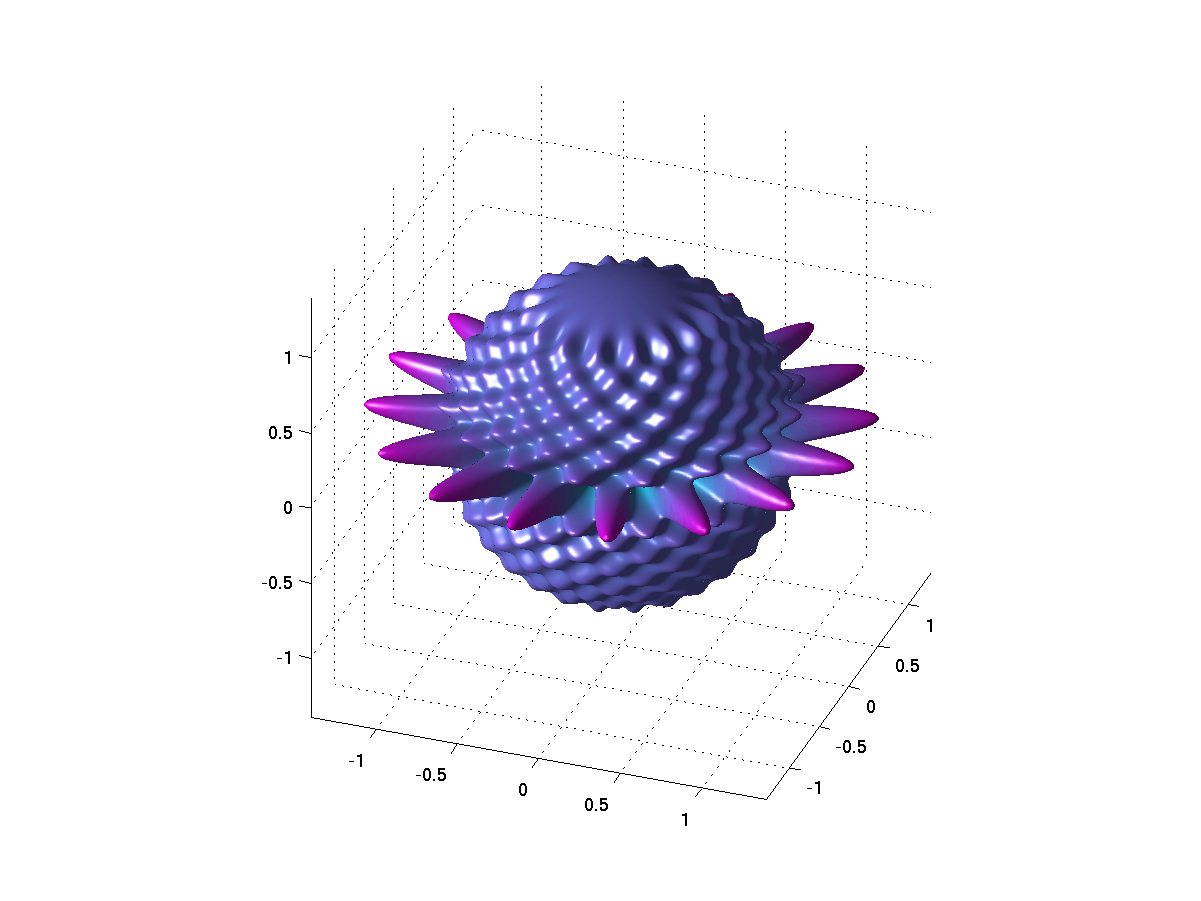}} \hspace{-12mm}
	\subfigure[$\psi_{32,1}$, $x_{32,1} = 0$]{\includegraphics[width=0.38\textwidth]{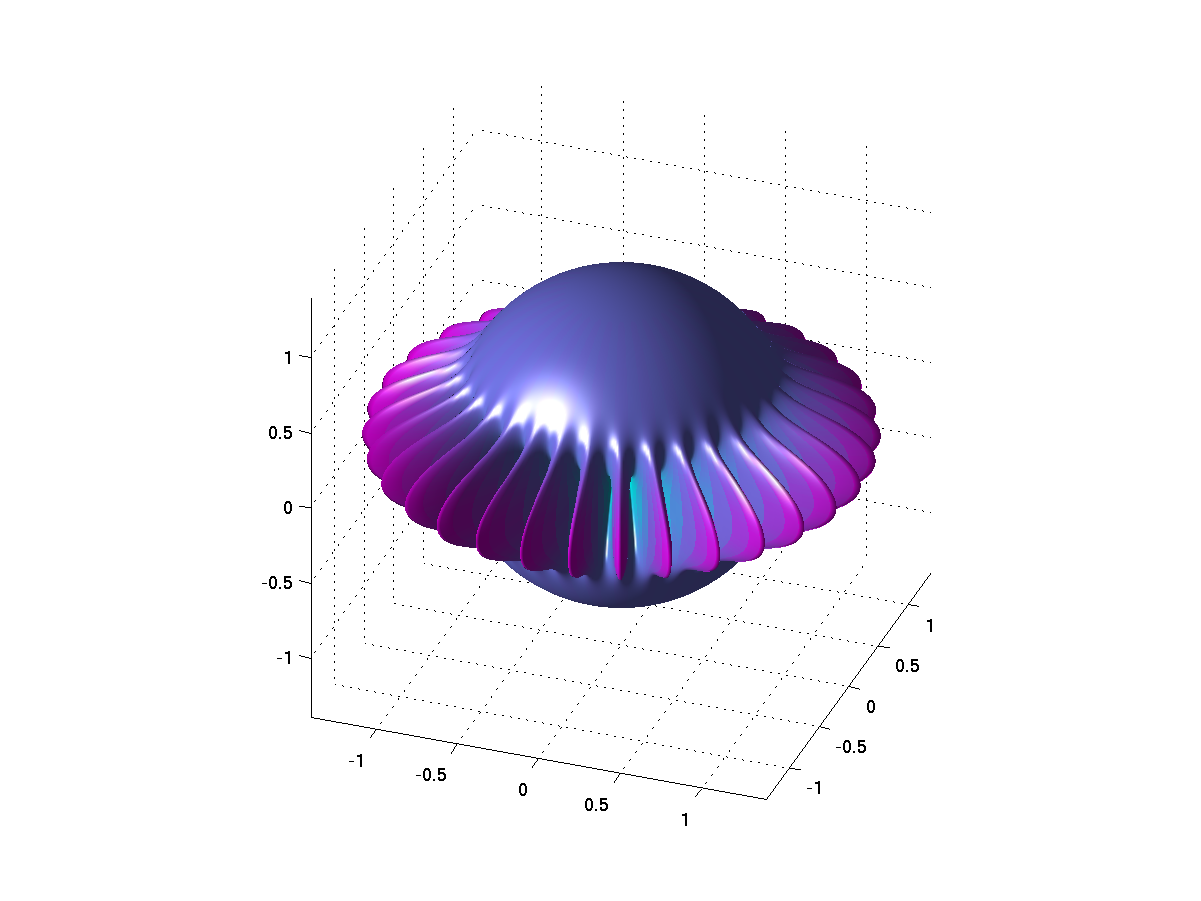}} 
	\caption{Real part of the eigenfunctions $\psi_{k,i}$ for $n=32$, $m= 0$ and their localization with respect to the north pole.}
	\label{fig:plot1}
 \end{figure}

Now, given a space-localized basis for the polynomial spaces $\Pi_n^m$, we want to analyse the decomposition of a polynomials $Q$ in the new basis. 
The next theorems illustrate that for a good approximation of a space-localized polynomial $Q$ only the eigenfunctions are needed that are located in the region where the mass of $Q$ is concentrated. 

\begin{thm} \label{thm:localizedapproximationMarkov}
Let $a >0$, $I_a^- := [-1,-1+a)$ and $I_a^+:=(1-a,1]$. If $Q\in \Pi_n^m$ with $ \Vert Q \Vert = 1$, the following error bounds hold:
\begin{align}
\label{eq:errorbound1}
	\Bigg \Vert &  Q- \sum_{k=-n}^n \ \sum_{i: x_{|k|,i} \in I_a^-} \langle Q, \psi_{k,i} \rangle \, \psi_{k,i} \Bigg \Vert^2 \leq \frac{1+\epsilon(Q)}{a}, \\
\label{eq:errorbound2}
	\Bigg \Vert &  Q- \sum_{k=-n}^n \ \sum_{i: x_{|k|,i} \in I_a^+ }\langle Q, \psi_{k,i} \rangle \, \psi_{k,i} \Bigg \Vert^2 \leq \frac{1-\epsilon(Q)}{a}. 
\end{align}
\end{thm}

\begin{proof}
Consider a polynomial $Q \in \Pi_n^m$ with $ \Vert Q \Vert = 1$. Since the eigenfunctions $\psi_{k,i}$ form an orthonormal basis of $\Pi_n^m$, the polynomial $Q$ can be written as 
\[ Q(\theta, \varphi) = \sum_{k = -n}^{n} \sum_{i=1}^{N_k} \langle Q, \psi_{k,i} \rangle \psi_{k,i}(\theta, \varphi).\]
Therefore,
\begin{align*}
\Bigg \Vert  Q - &\sum_{k=-n}^n \ \sum_{i: x_{|k|,i} \in I_a^-} \langle Q, \psi_{k,i} \rangle\, \psi_{k,i} \Bigg  \Vert^2
 = \Bigg \Vert  \sum_{k=-n}^n \ \sum_{i: x_{|k|,i} \in [-1+a,1]} \langle Q, \psi_{k,i} \rangle\, \psi_{k,i} \Bigg \Vert^2\\
 & = \sum_{k=-n}^n \ \sum_{i: x_{|k|,i} \in [-1+a,1]} \vert\langle Q, \psi_{k,i} \rangle \vert^2
\end{align*}
holds by the Pythagorean theorem. For the eigenvalues $x_{|k|,i}$ in  the interval $[-1+a,1] $
we have $(1+x_{|k|,i}) \geq a$. Consequently, we can derive the estimates 
\begin{align*}
\sum_{k=-n}^n  \sum_{i: x_{|k|,i} \in [-1+a,1] } \!\!\!  \vert\langle Q, \psi_{k,i} \rangle\vert^2 
 & \leq  \frac{1}{a}\sum_{k=-n}^n \ \sum_{i: x_{|k|,i} \in [-1+a,1]} \!\!\! \vert\langle Q, \psi_{k,i} \rangle\vert^2 (1+x_{|k|,i})\\
 & \leq   \frac{1}{a}\left(\sum_{k=-n}^n  \sum_{i=1}^{N_k} (1+x_{|k|,i}) \vert\langle Q, \psi_{k,i} \rangle\vert^2 \right). 
\end{align*} 
By the Pythagorean theorem, we have  
\[ \sum_{k=-n}^n \! \sum_{i=1}^{N_k} \vert \langle Q, \psi_{k,i} \rangle\vert^2 = \Vert Q \Vert = 1.\]
On the other hand, the spectral Theorem \ref{thm:spectraldecomposition} yields
\begin{align*}
\eps(Q) &= \langle \Pnm{} \MCos{} \Pnm{Q}, Q \rangle = \sum_{k=-n}^n \! \sum_{i=1}^{N_k} \vert\langle Q, \psi_{k,i} \rangle\vert^2 x_{|k|,i}.\end{align*}
Thus, we obtain the error bound \eqref{eq:errorbound1}.
The second error bound is derived from the fact that $(1-x_{|k|,i}) \geq a$ holds for eigenvalues $x_{|k|,i}\in [-1,1] \setminus I_a^+$. 
Analogous steps as for the first bound then lead to equation \eqref{eq:errorbound2}. 
\end{proof}

\begin{rem}
For a normalized polynomial $Q \in \Pi_n^m$, we can define the discrete density function $\rho$ on $\Rr$ by
\[ \rho(x) = \sum_{i,k: \eps(\psi_{k,i}) = x} |\langle Q, \psi_{k,i} \rangle|^2  \] 
supported on the set of eigenvalues $x_{|k|,i} \in (-1,1)$.
Then, the expectation value of a $\rho$-distributed random variable $X$ is given by $\eps(Q)$ and the statement of Theorem \ref{thm:localizedapproximationMarkov} corresponds to 
Markov's inequality (cf. \cite[p. 114]{Papoulis}) for the random variables $1+X$ and $1-X$. 
\end{rem}
Now, if we define the discrete variance 
\[ \var_\rho(Q) := \sum_{k,i} \vert\langle Q, \psi_{k,i} \rangle\vert^2 (x_{|k|,i}-\eps(Q))^2 = \sum_{k,i} \vert\langle Q, \psi_{k,i} \rangle\vert^2 (x_{|k|,i}^2-\eps(Q)^2) \]
and use Chebyshev's inequality (\cite[p. 114]{Papoulis}) for a $\rho$-distributed random variable,
we immediately get the following result.

\begin{cor} \label{cor:localizedapproximationChebyshev}
Let $a >0$, $Q\in \Pi_n^m$, $ \Vert Q \Vert = 1$ and $I_a = (\eps(Q) - a, \eps(Q) + a)$. Then, the following error bound holds
\begin{equation}
\label{eq:errorbound3}
	\Bigg \Vert  Q - \sum_{k=-n}^n \ \sum_{i: x_{|k|,i} \in I_a} \langle Q, \psi_{k,i} \rangle \, \psi_{k,i} \Bigg \Vert^2 \leq \frac{\var_\rho(Q)}{a^2}.
\end{equation}
\end{cor} 

In Theorem \ref{thm:localizedapproximationMarkov} and Corollary \ref{cor:localizedapproximationChebyshev} it is a priori not clear how many zeros $x_{|k|,i}$ are contained in the intervals $I_a$, $I_a^+$ or $I_a^-$. Therefore, we want to analyze now the distribution of the zeros $x_{|k|,i}$ on the interval $[-1,1]$. We give first an auxiliary result about particular operators related to the 
space-frequency operator $\Pnm{} \MCos{} \Pnm{}$. 

\begin{lem} \label{lemma-rankofoperators}
For $j = 1,2,3, \cdots$, the operators
\begin{align} \label{equation-differenceoperators1} 
\mathrm{A}_j &= (\Pnn{} \MCos{} \Pnn{})^j - \Pnn{} \MCos{}^j \Pnn{},\\
\mathrm{B}_j &= \Pnn{} \MCos{}^j \Pnn{} - \Pnm{} \MCos{}^j \Pnm{}, \\
\mathrm{C}_j &= (\Pnm{} \MCos{} \Pnm{})^j - (\Pnn{} \MCos{} \Pnn{})^j,
\end{align} 
have norm at most $2$. The operator $A_j$ is of rank at most $(2n+1-j)(j-1)$,  while the operators
$B_j$ and $C_j$ are of rank at most $(2m-1)(n+1)-m(m-1)$. 
\end{lem}

\begin{proof}
All the involved operators $\Pnm{} \MCos{}^j \Pnm{}$ and $(\Pnm{} \MCos{} \Pnm{})^j$, $j = 1,2,3, \cdots$, have norm smaller than $1$. Thus, the operators 
$\mathrm{A}_j$, $\mathrm{B}_j$ and $\mathrm{C}_j$ have operator norm at most $2$. The operators $\mathrm{A}_j$, $\mathrm{B}_j$ and $\mathrm{C}_j$ map 
the orthogonal complement $(\Pin^0)^\perp$ to zero. For the polynomial space $\Pin^0$, we use
the spherical harmonics $Y_l^k$, $0 \leq l \leq n$, $-l \leq k \leq l$, as a basis. We have $\mathrm{B}_j Y_l^k = \mathrm{C}_j Y_l^k = 0$ if $|k| \geq m$. Summing up all spherical harmonics $Y_l^k$ with
$|k| < m$, we get the upper bound $(2m-1)(n+1) - m^2+m$ for the rank of the operators $\mathrm{B}_j$ and $\mathrm{C}_j$ (for $m=0$, we get of course $0$). For the operators $\mathrm{A}_j$, we have
$\mathrm{A}_j Y_l^k = 0$ for all $Y_l^k$ with $l \leq n-j+1$. Summing up all possible spherical harmonics with index $n \geq l \geq n-j$, we get $(2n+1-j)(j-1)$ as an upper bound for the rank
of the operator $\mathrm{A}_j$. 
\end{proof}

Now, we can state the following weak limit for the distribution of the zeros $x_{|k|,i}$ as $n$ tends to infinity.

\begin{thm} \label{thm:weaklimit}
Let $f$ be a bounded, Riemann integrable function on $[-1,1]$, then
\begin{equation} \label{equation-weaklimit} \lim_{n\to \infty} \frac{1}{N_n^m} \sum_{k=-n}^n \sum_{i = 1}^{\Nk}  f(x_{|k|,i}) = \frac{1}{2}\int_{-1}^1 f(x) dx.\end{equation}
\end{thm}

\begin{proof}
According to Theorem \ref{thm:spectraldecomposition}, the functions $\psi_{k,i}$, $-n \leq k \leq n$, $1 \leq i \leq N_k$, 
form a complete set of eigenfunctions for the operator $\Pnm{} \MCos{} \Pnm{}$ on the space $\Pinm$. Thus, the functions 
$\psi_{k,i}$ are also eigenfunctions of the operators $(\Pnm{} \MCos{} \Pnm{})^j$, $j = 1,2,3, \cdots$, and we get:
\[ \tr \big( (\Pnm{} \MCos{} \Pnm{})^j\big) = \sum_{k=-n}^n \sum_{i = 1}^{\Nk} \langle (\Pnm{} \MCos{} \Pnm{})^j \psi_{k,i}, \psi_{k,i} \rangle 
= \sum_{k=-n}^n \sum_{i = 1}^{\Nk} x_{|k|,i}^{\; j}.\]
On the other hand, also the spherical harmonics $Y_l^k$, $m \leq l \leq n$, $-l \leq k \leq l$, form an orthonormal basis of $\Pinm$. Using the addition theorem of
the spherical harmonics (see \cite[Theorem 2]{Mueller}), we obtain the identity:
\begin{align*}
 \tr \big( \Pnm{} \MCos{}^j \Pnm{}\big) &= \sum_{l=m}^n \sum_{k=-l}^{l} \langle (\Pnm{} \MCos{}^j \Pnm{}) Y_l^k, Y_l^k \rangle \\
&= \frac{1}{4\pi} \sum_{l=m}^n \sum_{k=-l}^{l} \int_0^{2\pi} \int_{0}^{\pi} ( \cos \theta )^j |Y_l^k(\theta,\vph)|^2 \sin \theta d \theta d \vph \\
&= \frac{1}{4\pi} \sum_{l=m}^n \int_0^{2\pi} \int_{0}^{\pi} ( \cos \theta )^j \left( \sum_{k=-l}^{l} |Y_l^k(\theta,\vph)|^2 \right) \sin \theta d \theta d \vph \\
&= \sum_{l=m}^n \frac{2l+1}{2} \int_{0}^{\pi} ( \cos \theta )^j \sin \theta d \theta \\
&= \sum_{l=m}^n \frac{2l+1}{2} \int_{-1}^{1} x^j dx = \frac{N_n^m}{2} \int_{-1}^{1} x^j dx.
\end{align*}
Now, by Lemma \ref{lemma-rankofoperators}, we get the estimate
\begin{align*}
 & \left| \frac{1}{N_n^m}\sum_{k=-n}^n \sum_{i = 1}^{\Nk} x_{|k|,i}^{\; j} - \frac{1}{2} \int_{-1}^{1} x^j dx \right| = \frac{ |\tr \mathrm{A}_j + \tr \mathrm{B}_j + \tr \mathrm{C}_j| }{(n+1)^2-m^2}\\
 & \quad \leq \frac{|\tr A_1| + |\tr A_2| + |\tr A_3|}{(n+1)^2-m^2}\leq 4 \frac{(n+1)(j+2m-2)-\frac{j^2+1}{2}-m^2+m}{(n+1)^2-m^2}.
\end{align*}
For every fixed $j$, the term on the right hand side tends to zero as $n \to \infty$. In this way, we have shown equation \eqref{equation-weaklimit} for every polynomial on $[-1,1]$. The statement
for arbitrary bounded Riemann integrable functions $f$ on $[-1,1]$ follows from the theory of one-sided polynomial approximation developed by Freud (see \citep{Freud}).
\end{proof}

\begin{cor}
For $n \to \infty$, the distribution of the eigenvalues $x_{|k|,i}$ converges weakly to the uniform distribution on $[-1,1]$, i.e.,
for $-1 \leq a < b \leq 1$ we have
\[ \lim_{n \to \infty} \frac{\sharp \{(k,i): \; -n \leq k \leq n, \;1 \leq i \leq N_k,\; x_{|k|,i} \in [a,b]\}}{N_n^m} = \frac{b-a}{2}. \]
\end{cor}

\begin{rem}
The idea for the proofs of Lemma \ref{lemma-rankofoperators} and Theorem \ref{thm:weaklimit} is taken from the works of Simon (see \citep[Section 2.15]{Simon2011}, \cite{Simon2009}).
In these works, the corresponding weak limits are proven for orthogonal polynomials on the real line and on the unit circle. 
\end{rem}

\section{Computational considerations}

When applying the new localized basis for the analysis of functions on $\Szwei$, an important aspect is the numerical effort to compute the expansion coefficients as well 
as to reconstruct the function from the coefficients. In this section, we will show that due to the particular structure of the basis functions both can be done fast 
and efficiently from the expansion of the function in spherical harmonics. 

In a first step, we investigate the relation
\[Q = \sum_{l=m}^n \sum_{k=-l}^l c_{l,k} Y_l^k = \sum_{k=-n}^n \sum_{i=1}^{N_k} d_{k,i} \psi_{k,i}\]
between an expansion in spherical harmonics and an expansion in the new localized basis. We use the notation
\begin{align*}
\textbf{d}_{k} &:= (d_{k,1}, d_{k,2} \ldots, d_{k,N_k})^T,\quad -n \leq k \leq n,
\end{align*}
and consider the relation between the spherical harmonics $Y_l^k$ and the eigenfunctions $\psi_{k,i}$. By the spectral Theorem \ref{thm:spectraldecomposition}, we have for $-n \leq k \leq n$:
\begin{align*}
\psi_{k,i} &= \mathrm{T}_k \frac{\mathbf{v}_{|k|,i}}{\|\mathbf{v}_{|k|,i}\|_2} = \left( Y_{n-N_k+1}^k, \cdots, Y_n^k \right) \cdot \frac{\mathbf{v}_{|k|,i}}{\|\mathbf{v}_{|k|,i}\|_2},  \qquad 1 \leq i \leq N_k.
\end{align*}
Now, comparison of the two different expansions gives 
\begin{align*}
\textbf{c}_{k} &= \underbrace{\left( \frac{\mathbf{v}_{|k|,1}}{\|\mathbf{v}_{|k|,1}\|_2}, \cdots,  \frac{\mathbf{v}_{|k|,N_k}}{\|\mathbf{v}_{|k|,N_k}\|_2} \right)}_{\textbf{V}_k} \ddd_k, \quad -n \leq k \leq n.
\end{align*}
Since for fixed $k$ the eigenvectors $\mathbf{v}_{|k|,i}$, $1 \leq i \leq N_k$, of the symmetric matrices $\Jj\big({ \textstyle |k|}\big)_{n-|k|}^{m-|k|}$ and
$\Jj\big({ \textstyle |k|}\big)_{n-|k|}$, are pairwise orthogonal, the matrices $\textbf{V}_k$ are orthogonal and we get
\begin{equation} \label{eq:computationofcoefficients}
\textbf{c}_{k} = \textbf{V}_{k} \textbf{d}_{k}, \quad \textbf{d}_{k} = \textbf{V}_{k}^T \textbf{c}_{k}, \quad -n \leq k \leq n.
\end{equation}
Once all the eigenvectors of the matrices $\Jj\big({ \textstyle |k|}\big)_{n-|k|}^{m-|k|}$ and
$\Jj\big({ \textstyle |k|}\big)_{n-|k|}$ are computed, the coefficients $\textbf{d}_{k}$ of the localized 
basis can be computed from the expansion coefficients $\textbf{c}_{k}$ by 
a matrix-vector product in $(2 N_k - 1 ) N_k$ arithmetic operations. For all $2n+1$ involved blocks we then get a total
amount of $\frac13 (n-m+1)(4n^2+n(4m+5) + 3 + m - 8 m^2)$ arithmetic operations to conduct the change of basis.

Due to the particular structure of the transition matrices $\textbf{V}_{k}$ and $\textbf{V}_{k}^T$, 
the calculation of the coefficients $\textbf{d}_{k}$ can be accelerated considerably by using algorithms based
on the fast Fourier transform. To this end, we consider the entries of $\textbf{V}_{k}^T$ in more detail. 
To simplify the representations, we will only consider the case $m \leq |k| \leq n$.
In this case, we have
\begin{equation*}
\begin{pmatrix}
d_{k,1} \\ \vdots \\ d_{k,N_k}
\end{pmatrix}
=
\underbrace{\begin{pmatrix}
\kappa_{k,1} & & \\
 & \ddots & \\
 &  & \kappa_{k,N_k}
\end{pmatrix}}_{\mathbf{K}_k} 
\underbrace{\begin{pmatrix}
p_0^{(|k|)}(x_{|k|,1}) & \cdots & p_{n-|k|}^{(|k|)}(x_{|k|,1}) \\
\vdots & \ddots & \vdots \\
p_0^{(|k|)}(x_{|k|,N_k}) & \cdots & p_{n-|k|}^{(|k|)}(x_{|k|,N_k})
\end{pmatrix} }_{\mathbf{\tilde{V}}_k^T} \begin{pmatrix}
c_{1,k} \\ \vdots \\ c_{N_k,k}
\end{pmatrix}.
\end{equation*}
Now, we consider transition matrices $\mathbf{B}_k$ (see \citep[Section 1.4]{Keiner}, \citep{PottsSteidlTasche1998}) 
that describe the change of basis from ultraspherical polynomials $p_l^{(|k|)}(x)$ to Chebyshev polynomials $T_l(x) = \cos(l \arccos x)$ of the first kind, i.e.
\begin{equation*}
\underbrace{\begin{pmatrix}
T_0(x_{|k|,1}) & \cdots & T_{n-|k|}(x_{|k|,1}) \\
\vdots & \ddots & \vdots \\
T_0(x_{|k|,N_k}) & \cdots & T_{n-|k|}(x_{|k|,N_k})
\end{pmatrix} }_{\mathbf{F}_k} \mathbf{B}_k = 
\begin{pmatrix}
p_0^{(|k|)}(x_{|k|,1}) & \cdots & p_{n-|k|}^{(|k|)}(x_{|k|,1}) \\
\vdots & \ddots & \vdots \\
p_0^{(|k|)}(x_{|k|,N_k}) & \cdots & p_{n-|k|}^{(|k|)}(x_{|k|,N_k})
\end{pmatrix}
\end{equation*}
In this way, we can write the transforms in \eqref{eq:computationofcoefficients} as
\[\textbf{d}_{k} = \mathbf{K}_{k} \mathbf{F}_{k} \mathbf{B}_{k} \textbf{c}_{k} , \qquad \textbf{c}_{k} = \mathbf{B}_{k}^T \mathbf{F}_{k}^T \mathbf{K}_{k} \textbf{d}_{k}, 
\quad -n \leq k \leq n.\]
Using this representation, we can now use efficient algorithms to compute the single steps in a fast way. The matrix-vector product
$\mathbf{B}_{k} \textbf{c}_{k}$ can be computed efficiently using the fast polynomial transform described in \citep{PottsSteidlTasche1998} in $\Ord(N_k \log^2 N_k)$ 
arithmetic operations. Next, the multiplication with $\mathbf{F}_k$ can be conducted using a non-equispaced fast cosine transform
in $\Ord(N_k \log N_k)$ arithmetic operations (cf. \cite{FennPotts2005}). Finally, the application of $\mathbf{K}_{k}$ can be implemented cheaply (with $N_k$ arithmetic operations) 
by a point-wise vector-vector multiplication. In this way, 
the coefficients $\mathbf{d}_k$ can be computed from the coefficients $\mathbf{c}_k$ with a complexity of $\Ord(N_k \log^2 N_k)$. With $N_k \leq n-m+1$, we get for all $2n+1$ 
involved blocks a total complexity of 
\[ \Ord(n (n-m) \log^2 (n-m))\]
arithmetic operations for the change in the new localized basis. From the order of complexity, this corresponds to the 
complexity of fast algorithms computing the Fourier transform on $\Szwei$ (cf. \citep{KunisPotts2003}). Implementations 
of the fast polynomial transform, the fast non-equispaced cosine transform and the fast spherical Fourier transform can be found 
in the software package NFFT3 documented in \citep{KeinerKunisPotts2009} and the references therein. An example plot of a computed decomposition can be found in
Figure \ref{fig:plot2}.

\begin{figure}[H]
	\centering

	\subfigure[$f \in \Pi_{512}$]{\includegraphics[width=0.38\textwidth]{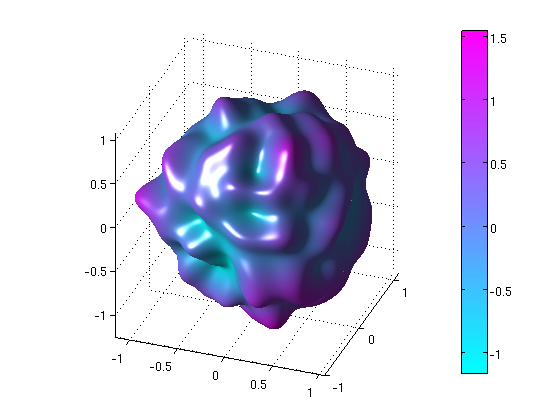}} \hspace{-12mm}
	\subfigure[$f_1 = \underset{k,i: x_{|k|,i} \in I}{\sum} d_{k,i} \psi_{k,i}$]{\includegraphics[width=0.38\textwidth]{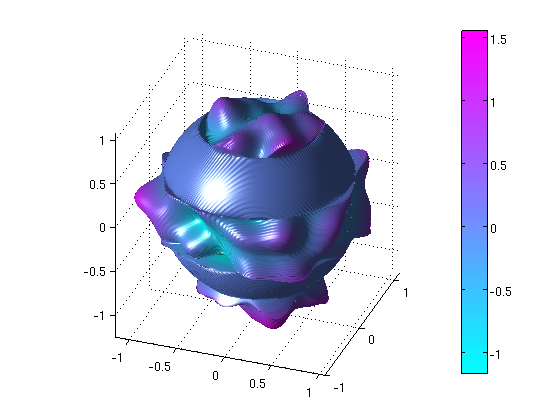}} \hspace{-12mm}
	\subfigure[$f_2 = \underset{k,i: x_{|k|,i} \notin I}{\sum} d_{k,i} \psi_{k,i}$]{\includegraphics[width=0.38\textwidth]{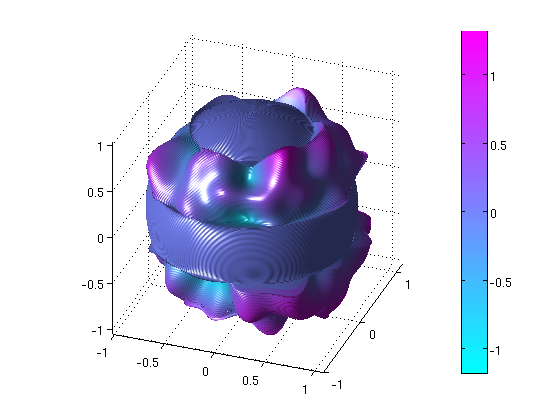}} \\

	\caption{Orthogonal decomposition of a function $f \in \Pi_{512}$ in two parts $f_1, f_2 \in \Pi_{512}$. 
	The set $I \subset [-1,1]$ is given by $[-1,-0.6] \cup [-0.2,0.2] \cup [0.6,1]$. The functions $f_1$ and $f_2$ are 
	localized in the regions $\{ p(\theta,\varphi) \in \Szwei:\; \theta \in \arccos I\}$ and $\{ p(\theta,\varphi) \in \Szwei:\; \theta \notin \arccos I\}$ of $\Szwei$, respectively.}
	\label{fig:plot2}
 \end{figure}


\begin{thebibliography}{10}

\bibitem{AlbertellaSansoSneeuw1999}
{\sc Albertella, A., Sans\`o, F., and Sneeuw, N.}
\newblock {Band-limited functions on a bounded spherical domain: The Slepian
  problem on the sphere}.
\newblock {\em J. Geod. 73}, 9 (1999), 436--447.

\bibitem{Chihara}
{\sc Chihara, T.~S.}
\newblock {\em An Introduction to Orthogonal Polynomials}.
\newblock Gordon and Breach, Science Publishers, New York, 1978.

\bibitem{ConradPrestin2002}
{\sc Conrad, M., and Prestin, J.}
\newblock Multiresolution on the {S}phere.
\newblock In {\em Summer School Lecture Notes on Principles of Multiresolution
  in Geometric Modelling, Munich, August 22-30, 2001\/} (2002), A.~Iske,
  E.~Quak, and M.~Floater, Eds., Springer, Heidelberg, pp.~165--202.

\bibitem{DaiXu}
{\sc Dai, F., and Xu, Y.}
\newblock {\em {Approximation Theory and Harmonic Analysis on Spheres and
  Balls}}.
\newblock Springer-Verlag, New York, 2013.

\bibitem{Erb2010}
{\sc Erb, W.}
\newblock Uncertainty principles on compact {R}iemannian manifolds.
\newblock {\em Appl. Comput. Harmon. Anal. 29}, 2 (2010), 182--197.

\bibitem{ErbDiss}
{\sc Erb, W.}
\newblock {\em Uncertainty Principles on Riemannian Manifolds}.
\newblock {Logos Verlag Berlin}, {dissertation}, Technical University Munich,
  2010.

\bibitem{erb2012}
{\sc Erb, W.}
\newblock Optimally space localized polynomials with applications in signal
  processing.
\newblock {\em J. Fourier Anal. Appl. 18}, 1 (2012), 45--66.

\bibitem{erb2013}
{\sc Erb, W.}
\newblock {An orthogonal polynomial analogue of the Landau-Pollak-Slepian
  time-frequency analysis}.
\newblock {\em Journal of Approximation Theory 166\/} (2013), 56--77.

\bibitem{ErbTookos2011}
{\sc Erb, W., and To\'okos, F.}
\newblock Applications of the monotonicity of extremal zeros of orthogonal
  polynomials in interlacing and optimization problems.
\newblock {\em Appl. Math. Comput. 217}, 9 (2011), 4771--4780.

\bibitem{FennPotts2005}
{\sc Fenn, M., and Potts, D.}
\newblock {Fast summation based on fast trigonometric transforms at
  nonequispaced nodes}.
\newblock {\em Numer. Linear Algebra Appl. 12\/} (2005), 161--169.

\bibitem{Freeden}
{\sc Freeden, W., Gervens, T., and Schreiner, M.}
\newblock {\em Constructive Approximation on the Sphere with Applications to
  Geomathematics}.
\newblock Oxford Science Publications, Clarendon Press, Oxford, 1998.

\bibitem{Freud}
{\sc Freud, G.}
\newblock {\em {Orthogonale Polynome}}.
\newblock {Birkh\"auser Verlag, Basel und Stuttgart}, 1969.

\bibitem{Gautschi}
{\sc Gautschi, W.}
\newblock {\em Orthogonal Polynomials: Computation and Approximation}.
\newblock Oxford University Press, Oxford, 2004.

\bibitem{GohGoodman2004-2}
{\sc Goh, S.~S., and Goodman, T.~N.}
\newblock Uncertainty principles and optimality on circles and spheres.
\newblock In {\em Advances in constructive approximation: Vanderbilt 2003.
  Proceedings of the international conference, Nashville, TN, USA, May 14--17,
  2003\/} (2004), M.~Neamtu and E.~B. Saff, Eds., Nashboro Press, Brentwood,
  TN, Modern Methods in Mathematics, pp.~207--218.

\bibitem{GrafarendAwange}
{\sc Grafarend, E.~W., and Awange, J.~L.}
\newblock {\em Applications of Linear and Nonlinear Models: Fixed Effects,
  Random Effects, and Total Least Squares}.
\newblock Springer-Verlag, Berlin-Heidelberg, 2012.

\bibitem{GruenbaumLonghiPerlstadt1982}
{\sc Grünbaum, F.~A., Longhi, L., and Perlstadt, M.}
\newblock {Differential operators commuting with finite convolution integral
  operators: Some non-Abelian examples.}
\newblock {\em SIAM J. Appl. Math. 42\/} (1982), 941--955.

\bibitem{Ismail}
{\sc Ismail, M.~E.}
\newblock {\em Classical and Quantum Orthogonal Polynomials in One Variable}.
\newblock Cambridge University Press, Cambridge, 2005.

\bibitem{Keiner}
{\sc Keiner, J.}
\newblock {\em Fast polynomial transform}.
\newblock {Logos Verlag Berlin}, {dissertation}, University of Luebeck, 2011.

\bibitem{KeinerKunisPotts2009}
{\sc Keiner, J., Kunis, S., and Potts, D.}
\newblock {NFFT 3 - a software library for various nonequispaced fast Fourier
  transforms}.
\newblock {\em ACM Trans. Math. Software 36\/} (2009), 1--30.

\bibitem{KunisPotts2003}
{\sc Kunis, S., and Potts, D.}
\newblock {Fast spherical Fourier algorithms}.
\newblock {\em J. Comput. Appl. Math. 161\/} (2003), 75 -- 98.

\bibitem{Fernandez2007}
{\sc La\'in~Fern\'andez, N.}
\newblock Optimally space-localized band-limited wavelets on
  $\mathbb{S}^{q-1}$.
\newblock {\em J. Comput. Appl. Math. 199}, 1 (2007), 68--79.

\bibitem{LandauPollak1961}
{\sc Landau, H., and Pollak, H.}
\newblock {Prolate spheroidal wave functions, Fourier analysis and uncertainty,
  II.}
\newblock {\em Bell System Tech. J. 40\/} (1961), 65--84.

\bibitem{LandauPollak1962}
{\sc Landau, H., and Pollak, H.}
\newblock {Prolate spheroidal wave functions, Fourier analysis and uncertainty,
  III: The dimension of the space of essentially time- and band-limited
  signals.}
\newblock {\em Bell System Tech. J. 41\/} (1962), 1295--1336.

\bibitem{LandauWidom1980}
{\sc Landau, H., and Widom, H.}
\newblock {Eigenvalue distribution of time and frequency limiting.}
\newblock {\em J. Math. Anal. Appl. 77\/} (1980), 469--481.

\bibitem{Michel2011}
{\sc Michel, V.}
\newblock Optimally localized approximate identities on the 2-sphere.
\newblock {\em Numerical Functional Analysis and Optimization 32\/} (2011),
  877--903.

\bibitem{Michel}
{\sc Michel, V.}
\newblock {\em {Lectures on Constructive Approximation - Fourier, Spline, and
  Wavelet Methods on the Real Line, the Sphere, and the Ball}}.
\newblock Birkhäuser Verlag, Boston, 2012.

\bibitem{Miranian2004}
{\sc Miranian, L.}
\newblock {Slepian functions on the sphere, generalized Gaussian quadrature
  rule.}
\newblock {\em Inverse Problems 20}, 3 (2004), 877--892.

\bibitem{Mueller}
{\sc Müller, C.}
\newblock {\em Spherical Harmonics}.
\newblock Lecture Notes in Mathematics 17, Springer-Verlag,
  Berlin-Heidelberg-New York, 1966.

\bibitem{NarcovichWard1996}
{\sc Narcowich, F.~J., and Ward, J.~D.}
\newblock Nonstationary wavelets on the $m$-sphere for scattered data.
\newblock {\em Appl. Comput. Harmon. Anal. 3}, 4 (1996), 324--336.

\bibitem{Papoulis}
{\sc Papoulis, A.}
\newblock {\em Probability, Random Variables, and Stochastic Processes},
  third~ed.
\newblock McGraw-Hill, New York, 1991.

\bibitem{Polyakov}
{\sc Polyakov, A.}
\newblock {\em {Local Basis Expansions for Linear Inverse Problems}}.
\newblock Ph.d. thesis, New York University, 2002.

\bibitem{PottsSteidlTasche1998}
{\sc Potts, D., Steidl, G., and Tasche, M.}
\newblock Fast algorithms for discrete polynomial transforms.
\newblock {\em Math. Comput. 67\/} (1998), 1577--1590.

\bibitem{ReedSimon1}
{\sc Reed, M., and Simon, B.}
\newblock {\em {Methods of modern mathematical physics. Vol. I: Functional
  analysis}}.
\newblock {Academic Press, New York}, 1980.

\bibitem{RoeslerVoit1997}
{\sc R{\"o}sler, M., and Voit, M.}
\newblock An uncertainty principle for ultraspherical expansions.
\newblock {\em J.~Math.~Anal.~Appl. 209\/} (1997), 624--634.

\bibitem{RudinFunctionalAnalysis}
{\sc Rudin, W.}
\newblock {\em Functional Analysis}.
\newblock MacGraw-Hill, New York, 1973.

\bibitem{Simon2009}
{\sc Simon, B.}
\newblock Weak convergence of cd kernels and applications.
\newblock {\em Duke Mathematical Journal 146}, 2 (2009), 305 -- 330.

\bibitem{Simon2011}
{\sc Simon, B.}
\newblock {\em {Szeg\H o's theorem and its descendants. Spectral theory for
  $L^2$ perturbations of orthogonal polynomials.}}
\newblock {Princeton University Press, Princeton, NJ}, 2011.

\bibitem{Simons2012}
{\sc Simons, F.~J.}
\newblock {Slepian functions and their use in signal estimation and spectral
  analysis}.
\newblock In {\em Handbook of Geomathematics\/} (2010), {W. Freeden, M.Z.
  Nashed and T. Sonar}, Ed., {Berlin: Springer}, pp.~893--920.

\bibitem{SimonsDahlenWieczorek2006}
{\sc Simons, F.~J., Dahlen, F., and Wieczorek, M.~A.}
\newblock {Spatiospectral concentration on a sphere}.
\newblock {\em SIAM Rev. 48}, 3 (2006), 504--536.

\bibitem{Slepian1964}
{\sc Slepian, D.}
\newblock {Prolate spheroidal wave functions, Fourier analysis and uncertainty,
  IV: Extensions to many dimensions; generalized prolate spheroidal functions}.
\newblock {\em Bell System Tech. J. 43\/} (1964), 3009--3057.

\bibitem{Slepian1978}
{\sc Slepian, D.}
\newblock {Prolate spheroidal wave functions, Fourier analysis, and
  uncertainty, V: The discrete case}.
\newblock {\em Bell System Tech. J. 57\/} (1978), 1371--1430.

\bibitem{SlepianPollak1961}
{\sc Slepian, D., and Pollak, H.~O.}
\newblock {Prolate spheroidal wave functions, Fourier analysis and uncertainty,
  I}.
\newblock {\em Bell System Tech. J. 40\/} (1961), 43--63.

\bibitem{Szegoe}
{\sc {Szeg\H o}, G.}
\newblock {\em Orthogonal Polynomials}.
\newblock American Mathematical Society, Providence, Rhode Island, 1939.

\bibitem{Werner}
{\sc Werner, D.}
\newblock {\em Funktionalanalysis}, seventh~ed.
\newblock Springer-Verlag, Berlin, 2011.

\bibitem{WieczorekSimons2005}
{\sc Wieczorek, M.~A., and Simons, F.~J.}
\newblock {Localized spectral analysis on the sphere}.
\newblock {\em Geophys. J. Int. 162\/} (2005), 655--675.

\end{thebibliography}
\end{document}